%% file: 0paper.tex
\documentclass[a4paper]{amsart}
\usepackage[english]{babel} 
\usepackage[utf8]{inputenc}
\usepackage[T1]{fontenc}

\usepackage{amsmath,amsfonts,amssymb,amsthm}
\usepackage{xfrac}
\usepackage[longnamesfirst,numbers]{natbib}
\usepackage[margin]{fixme}
\usepackage{setspace}
\usepackage{enumitem}         
\setlist[itemize] {label=$\pmb\triangleright$}
\usepackage{lscape}
\usepackage{stmaryrd} 

\usepackage{color}
\definecolor{darkgreen}{rgb}{0,0.55,0}
\usepackage{hyperref}
\hypersetup{pdfstartview={Fit},    
    pdftitle={A Stefan-type stochastic moving boundary problem},    
    pdfauthor={Martin Keller-Ressel, Marvin S. Mueller},     
    pdfsubject={},   
    pagebackref=true,
    colorlinks=true,
    linkcolor=blue,
    citecolor=darkgreen}

\newtheorem{thm}{Theorem}[section]

\newtheorem{cor}[thm]{Corollary}
\newtheorem{lem}[thm]{Lemma}

\theoremstyle{definition}
\newtheorem{defn}[thm]{Definition}
\newtheorem{ass}[thm]{Assumption}
\theoremstyle{remark}

\newtheorem{rmk}[thm]{Remark}
\newtheorem{example}[thm]{Example}

\numberwithin{equation}{section}

\input{makrosMathBasics}

\input{makrosPartialDerivatives}

\input{makrosStochastic}

\input{makrosSemigroups}

\newcommand\HS{\operatorname{HS}}
\newcommand\DA{\dom(\cA)}

\renewcommand\H{\mathfrak H}
\renewcommand\L{\mathfrak L}

\newcommand\HTp{\widehat{\mathcal H}_{T,p}}

\newcommand\Kfix{\mathcal{K}}
\newcommand\normTp[1]{\norm{#1}{T,2p}}
\newcommand\cI{\mathcal I}

\newcommand\trafo{F} 

\newcommand{\bmu}{{\overline \mu}}
\newcommand{\bsigma}{{\overline \sigma}}

\newcommand{\llbrak}{\llbracket}
\newcommand{\rrbrak}{\rrbracket}

\title{A Stefan-type stochastic moving boundary problem}

\author{Martin Keller-Ressel}
\address{Institut f\"ur Math. Stochastik, TU Dresden, Germany}
\email[M. ~Keller-Ressel]{martin.keller-ressel@tu-dresden.de}

\author{Marvin S. M\"uller}
\email[M. S.~M\"uller]{marvin.s.mueller@tu-dresden.de}

\thanks{The authors acknowledge funding from the German Research Foundation (DFG) under grants ZUK 64 and RTG 1845.}

\date{\today}
\begin{document}
\begin{abstract}
Motivated by applications in economics and finance, in particular to the modeling of limit order books, we study a class of stochastic second-order PDEs with non-linear Stefan-type boundary interaction. To solve the equation we transform the problem from a moving boundary problem into a stochastic evolution equation with fixed boundary conditions. Using results from interpolation theory we obtain existence and uniqueness of local strong solutions, extending results of Kim, Zheng and Sowers.  In addition, we formulate conditions for existence of global solutions and provide a refined analysis of possible blow-up behavior in finite time.
\end{abstract}

\maketitle
\tableofcontents
\section{Introduction}
\input{intro}
\section{A stochastic moving boundary problem}
\label{sec:sfbp}
\input{stochfreebdryD}

\section{Solving a stochastic evolution equation}
\label{sec:seeqs}
\input{stocheveqDA}

\section{The fixed boundary problem as a stochastic evolution equation}
\label{sec:proofs} 
\input{proofEUStopping}

\section{Transformation from moving to fixed boundary}
\label{sec:trafo}
\input{itoformula2a_classic}

\begin{appendix}
\input{nemytskii}
\input{noisered}
\end{appendix}

\bibliographystyle{alpha}
\bibliography{litPaper}
\end{document}

%% file: makrosMathBasics.tex
\renewcommand\*{\cdot}
\newcommand\1{\mathbf{1}}

\newcommand\R{\mathbb{R}}
\newcommand\C{\mathbb{C}}
\newcommand\N{\mathbb{N}}

\providecommand{\setc}[2]{\left\{ #1 \cond #2\right\}}
\providecommand{\abs}[1]{\left\lvert#1\right\rvert}
\providecommand{\norm}[2]{\left\lVert#1\right\rVert_{#2}}
\newcommand{\scal}[3]{\left\langle #1, #2\right\rangle_{#3}}

\newcommand{\borel}[1]{\mathfrak B\left(#1\right)}

\newcommand{\RR}{\mathbb{R}}

\newcommand{\PP}{\mathbb{P}}

\newcommand{\NN}{\mathbb{N}}

\newcommand{\EE}{\mathbb{E}}

\newcommand{\cA}{\mathcal{A}}
\newcommand{\cB}{\mathcal{B}}

\newcommand{\cH}{\mathcal{H}}

\newcommand{\cP}{\mathcal{P}}

\newcommand{\cC}{\mathcal{C}}

%% file: makrosPartialDerivatives.tex
\newcommand\ddt{\tfrac{\partial}{\partial t}}

\newcommand\ddx{\tfrac{\partial}{\partial x}}
\newcommand\ddxx{\tfrac{\partial^2}{\partial x^2}}
\newcommand\ddy{\tfrac{\partial}{\partial y}}

\newcommand\ddz{\tfrac{\partial}{\partial z}}

%% file: makrosStochastic.tex
\newcommand\F{\mathcal{F}}

\renewcommand\d{\,\operatorname{d}\hspace{-0.05cm}}

\providecommand\cond{\,\middle\vert\,}
\renewcommand\cond{\,\middle\vert\,}

\renewcommand{\P}[1]{\mathbb{P}\left[#1\right]}                     
\newcommand{\E}[1]{\mathbb{E}\left[#1\right]}                     


%% file: makrosSemigroups.tex
\newcommand\Id{\operatorname{Id}}                               
\newcommand\dom{\mathcal{D}}                                      

%% file: intro.tex
Moving boundary problems allow to model multi-phase systems with separating boundaries evolving in time. Typically, the evolution of the free interface is strongly coupled with the evolution of the whole system. A classical example is the so called Stefan problem introduced in 1888 by Josef Stefan~\cite{stefanEis}, which describes the evolution of temperature $v(t,x)$ in a system of water and ice. In one space dimension it reads as
\begin{gather}
  \label{eq:StefanProb}
  \begin{split}
    \ddt v(t,x) &= \eta_i \ddxx v(t,x),\quad  x> x_*(t), \\ 
    \ddt v(t,x) &= \eta_w \ddxx v(t,x),\quad x< x_*(t),\\
    v(t,x_*(t)) &= 0,
  \end{split}
\end{gather}
where $\eta_w$ and $\eta_i$ are the thermal diffusivities of ice and water, and $x_*(t)$ is the position of the interface between the two phases. The evolution of the interface is governed by the so-called Stefan condition
\begin{equation}
  \label{eq:StefanCond}
  \ddt x_*(t) = \varrho \cdot \left( \ddx v(t,x_*(t)+) - \ddx v(t,x_*(t)-)\right),  \qquad \varrho > 0.
\end{equation}
This problem and various extensions have been studied extensively in the second half of the 20th century, see~\cite{stefanSurvey} for a review of the literature. For classical solutions of semi-linear extensions of~\eqref{eq:StefanProb} see e.\,g.~\cite{fpNonlinear},~\cite{lunardiMovingBoundary}. In addition to the theory of classical and weak solutions, the corresponding evolution equations have been studied in the framework of maximal $L^p$-regularity, see~\cite{pruessGibbsCorr},~\cite{pruessClassical} and references therein. 
Compared to the deterministic case, stochastic partial differential equations with free or moving interface have received much less attention. One exception is \cite{BarbuDaPratoStefan} where Barbu and da Prato show existence of a solution and an invariant ergodic measure for the linear problem~\eqref{eq:StefanProb} with additive noise in multiple dimensions. 

More recently, both deterministic and stochastic moving boundary problems have been applied in economics and finance to dynamic models of trading, in particular to models of so-called (electronic) limit order books where orders of buyers and sellers participating in stock exchanges are stored, see e.g. \cite{lasrylions,zhengPhD,bouchaudReactionDiffusion,bhqFunctional}. In such models, the space coordinate $x$ typically corresponds to price (usually on logarithmic scale), and the quantity $v(t,x)$ to the density of buyers or sellers willing to commit to a transaction at time $t$ for the price $x$. Buyers are recorded with positive sign and sellers with negative sign, such that the two phases of the system distinguish buyers from sellers. Of particular interest is the evolution of the separating boundary, which corresponds to the marginal price at which both sellers are currently willing to sell and buyers are willing to buy. Zheng~\cite{zhengPhD} for example proposes the following stochastic moving boundary problem as a model for dynamic trading in a limit order book:
\begin{gather}
  \label{eq:ZhengProb}
  \begin{split}
    \ddt v(t,x) &= \eta_s \ddxx v(t,x) + \sigma_s(|x - x_*(t)|) d\xi_t(x),\quad  x> x_*(t), \\ 
    \ddt v(t,x) &= \eta_b \ddxx v(t,x)+ \sigma_b(|x - x_*(t)|) d\xi_t(x) ,\quad x< x_*(t),\\
    v(t,x_*(t)) &= 0,
  \end{split}
\end{gather}
where the subscripts $b$ and $s$ correspond to buyer and seller respectively and $d\xi_t(x)$ is Gaussian noise. The evolution of the interface is governed by the linear Stefan condition  \eqref{eq:StefanCond}.  The accompanying mathematical theory is developed in \cite{sowersEtAl, sowersEtAl2} and numerical analysis in \cite{sowersNumeric}. Other examples can be found in Lasry and Lions~\cite{lasrylions} where a free boundary model for price formation under negotiation is introduced in a mean-field game setting. Another part of the literature derives SPDE models for limit order books as functional limits from discrete queuing models of orders that arrive and then are filled or cancelled. For examples of this approach see e.\,g. \cite{bhqFunctional} where a parabolic SPDE as a model for the order book is obtained in the limit. In addition, there is a series of papers by Bouchaud et al.~\cite{bouchaudReactionDiffusion},~\cite{bouchaudSMBP} with PDE and SPDE models observed as limiting equations of particle models.

In this paper we study a Stefan-type stochastic moving boundary problem, which can be considered an extension of \eqref{eq:ZhengProb} and of the theory developed in \cite{sowersEtAl} with several important differences in scope and methodology:

\begin{itemize}
\item Instead of the homogeneous linear stochastic Stefan problem, we allow for a more general drift coefficient and in particular a non-linear boundary condition replacing \eqref{eq:StefanCond}. Recent empirical studies of the dependency of price change on the imbalance of the order book (see \cite{contImbalance} and ~\cite{liptonImbalance}) suggest linear behaviour for balanced order books and non-linear behaviour when imbalance is large. 
\item In addition to mild and weak solutions as in ~\cite{sowersEtAl} we obtain solutions in the analytically strong sense and make the transformation from free to fixed boundary, that is introduced in a deterministic setting in \cite{lunardiMovingBoundary} and used in \cite{sowersEtAl}, rigorous in a stochastic setting.
\item We combine tools from the SPDE framework of da Prato and Zabczyk (cf.~\cite{dPZinf}) with results from interpolation theory, which allows for greater generality and avoids direct computations using the heat kernel as in ~\cite{sowersEtAl}
\end{itemize}

%% file: stochfreebdryD.tex
\subsection{Problem formulation}
Our goal is to establish a framework for solving stochastic moving boundary problems of the type
\begin{gather}
  \begin{split}
    \d v(t,x) &= \left[\eta_+ \ddxx v +  \mu_+\left(x-x_*(t), v, \ddx v \right) \right] \d t\\
    & \qquad \qquad\qquad\qquad\qquad+ \sigma_+\left(x-x_*(t), v\right)\d \xi_t (x), \quad  x > x_*(t),\\
    \d v(t,x) &= \left[\eta_- \ddxx v +\mu_-\left(x-x_*(t), v, \ddx v \right) \right] \d t\\
    &\qquad\qquad\qquad\qquad\qquad + \sigma_-\left(x-x_*(t), v\right) \d \xi_t(x), \quad  x < x_*(t),\\
    \intertext{with the moving boundary $x_*(t)$ governed by}
    \ddt x_*(t) &= \varrho\Big(\ddx v(t, x_*(t)+), \ddx v(t,x_*(t)-)\Big),
  \end{split}\label{eq:fbp}
\end{gather}
for $t\geq 0$, $x\in \R$, with Dirichlet boundary conditions at $x_*$, i.\,e.,
\begin{gather}
  \begin{split}
    v(t,x_*(t)+) &= 0,\\
    v(t,x_*(t)-) &= 0,
  \end{split}\label{eq:bc}\tag{BC}
\end{gather}
for $t\geq 0$. The coefficients are functions $\mu_\pm: \R^3\rightarrow \R$, $\sigma_\pm: \R^2\rightarrow \R$, and real numbers $\eta_\pm>0$. We denote by $\xi$ the spatially colored noise given by 
\begin{equation}
  \label{eq:colorednoise}
  \xi_t(x) := \int_0^t T_\zeta\d W_s (x),\qquad T_\zeta w(x) := \int_\R \zeta(x,y) w(y) \d y,\quad x\in \R,
\end{equation}
for some integral kernel $\zeta: \R^2\rightarrow \R$ and a cylindrical Wiener process $W$ on the Hilbert space $U=L^2(\R)$ with covariance operator identity. As usual, $W$ lives on a filtered probability space $(\Omega, \F, (\F_t)_{t \geq 0}, \PP)$.
For each $t \ge 0$ we require that $x \mapsto v(t,x)$ is continuously differentiable on $(-\infty, x_*(t))$ as well as on $(x_*(t),\infty)$, such that all first derivatives appearing in \eqref{eq:fbp} can be understood in the classical sense. The second derivative should be considered a weak derivative and a suitable function space for $v$ as well as the precise notion of `solution' to \eqref{eq:fbp} will be defined below.\\ 




We now make precise what we understand by a solution to the stochastic moving boundary problem \eqref{eq:fbp}. 
In general, solutions to the moving boundary problem may be local, i.e. only exist up to a stopping time $\tau$. To formalize this, it will be convenient to work with stochastic intervals. Given two stopping times $\varsigma \le \tau$ the stochastic interval $\llbrak \varsigma,\tau  \rrbrak$ is defined as
\[\llbrak \varsigma,\tau  \rrbrak := \left\{ (\omega,t) \in \Omega \times \R_+: \varsigma(\omega) \le t \le \tau(\omega) \right\}. \]
By using strict inequalities also the open stochastic interval $\rrbrak \varsigma,\tau  \llbrak$ and half-open analogues can be defined. As usual in a probabilistic setting we will soon `drop the omega' and say e.g. that for two stochastic processes $X$ and $Y$ the equality $X_t = Y_t$ holds for all $t \in \llbrak 0, \tau \rrbrak$, when we mean that $X_t( \omega) = Y_t(\omega)$ for $\PP$-almost all $\omega$ and all $t$ such that $(\omega,t)\in \llbrak 0,\tau\rrbrak$, i.\,e. 
\[ \P{X_t = Y_t, \forall\, t\leq \tau} = 1. \]

To formalize the moving frame for the moving boundary problem we define for $x \in \RR$ the function space
\begin{equation}\label{eq:moving_frame}
  \Gamma(x) := \left\{v: \R \to \R: \left.v\right|_{\R\setminus\{x\}} \in H^2(\R\setminus\{x\})\cap H^1_0(\R\setminus\{x\})\right\},
\end{equation}
where $H^1_0$ and $H^2$ are the usual Sobolev spaces. Note that due to the Sobolev embeddings, any function $v$ in $\Gamma(x)$ can be identified with an element of $L^2(\R)$. 

Finally, using the notation from \eqref{eq:fbp} we introduce the functions $\bmu: \R^4\rightarrow \R$, $\bsigma: \R^2\rightarrow \R$,
\begin{equation}
  \label{eq:bmubsigma}
  \bmu(x,v,v',v''):=
  \begin{cases}
    \eta_+ v'' + \mu_+(x,v,v'),  & \quad x>0,\\
    \eta_- v'' + \mu_-(x,v,v'), & \quad  x<0,
  \end{cases}
  \quad
  \bsigma(x,v) := 
  \begin{cases}
    \sigma_+(x,v), & \quad  x>0,\\
    \sigma_-(x,v), & \quad  x<0.
  \end{cases}
\end{equation}

\begin{defn}\label{df:fbp}
  A \emph{local solution} of the stochastic moving boundary problem~\eqref{eq:fbp} on the stochastic interval $\llbrak 0, \tau \llbrak$, with initial data $v_0$ and $x_0$, is a couple $(v,x_*)$ of stochastic processes, where 
  \begin{equation*}
    (v,x_*): \llbrak 0, \tau \llbrak \to \bigcup_{x \in \R} \left( \Gamma(x)\times \{x\} \right) \; \subseteq \;  L^2(\R)\times \R.
  \end{equation*}
  such that $(v,x_*)$ is predictable as an $L^2(\R)\times\R$-valued process, and
  \begin{align*}
    v(t) - v_0 &= \int_0^t \bmu(.-x_*(s), v(s), \ddx v(s), \ddxx v(s))\d s\\
    &\qquad\qquad+ \int_0^t \bsigma(.-x_*(s), v(s))\d\xi_s(.),\\
    \ddt x_*(t)  &=  \varrho\left(\ddx v(t, x_*(t)+), \ddx v(t,x_*(t)-)\right),\\
    x_*(0) &=x_0
  \end{align*}
  holds on $\llbrak 0, \tau \llbrak$. The first equality is an equality in $L^2(\R)$; the first integral is a Bochner integral in $L^2(\R)$, and the second one a stochastic integral in $L^2(\R)$.\\
  The solution is called global, if $\tau = \infty$ and the interval $\llbrak 0, \tau \llbrak$ is called maximal if there is no solution of \eqref{eq:fbp} on a larger stochastic interval.
\end{defn}
 
 \subsection{Assumptions and main results}\label{Sub:Assumptions}

We introduce the following assumptions on the coefficients appearing in \eqref{eq:fbp}
\begin{ass}\label{a:mu} The functions $\mu_\pm$ are continuously differentiable and
  \begin{enumerate}[label=(\roman*)]
  \item\label{ai:mugrowths} there exist $a \in L^2(\R_+)$, $b$, $\tilde b \in L^\infty_{loc}(\R^2; \R)$ such that for all $x, y, z\in \R$
    \[ \abs{\mu_\pm (x,y,z)} + \abs{\ddx \mu_\pm (x,y,z)} \leq  a(|x|) + b(y,z)\left(\abs{y} + \abs{z}\right),\]
    and
    \[ \abs{\ddy \mu_\pm(x,y,z)} +  \abs{\ddz\mu_\pm (x,y,z)} \leq \tilde b(y,z), \]
  \item\label{ai:mulip} $\mu_\pm$ and their partial derivatives (in $x$, $y$ and $z$) are locally Lipschitz with Lipschitz constants independent of $x\in \R$.
  \end{enumerate}
\end{ass}

\begin{ass}\label{a:sigma}
  The functions $\sigma_\pm$ are twice continuously differentiable and 
  \begin{enumerate}[label=(\roman*)]
  \item\label{ai:sigmagrowths} For every multi-index $I = (i,j) \in \N^2$ with $\abs{I} \leq 2$ there exist $a_I\in L^2(\R_+)$ and $b_I \in L^\infty_{loc}(\R, \R_+)$ such that
    \[  \abs{\tfrac{\partial^{\abs{I}}}{\partial x^{i}\partial y^{j}} \sigma(x,y)} \leq  \begin{cases}   a_I(|x|) + b_I(y)\abs{y}, & j = 0, \\ b_I (y), & j \neq 0. \end{cases}\]
  \item\label{ai:sigmalip} $\sigma_\pm$ and their partial derivatives (in $x$, $y$ and $z$) are locally Lipschitz with Lipschitz constants independent of $x\in \R$.  
  \item\label{ai:sigmabc} $\sigma_\pm$ satisfy the boundary condition
    \begin{equation}
      \label{eq:sigmaD}
      \sigma_\pm(0,0) = 0. 
    \end{equation}
  \end{enumerate}
\end{ass}

\begin{rmk} Later on, certain Nemytskii operators will be defined through $\mu_\pm$ and $\sigma_\pm$ and the assumptions made above can be traced back to requirements on the regularity of these operators, see Appendix~\ref{A:nem} and \ref{A:noise}. Also note that if $\sigma_+$ or $\sigma_-$ is independent of $x\in \R$, then, in Assumption~\ref{a:sigma}, part~\ref{ai:sigmabc} is a consequence of part~\ref{ai:sigmagrowths}.
\end{rmk}

\begin{ass}\label{a:rho}
  $\varrho:\R^2\to\R$ is locally Lipschitz continuous. More precisely, for all $N\in \N$ there exists an $L_{\varrho,N}$ such that \begin{equation*}
    \abs{\varrho(y) - \varrho(\tilde y)} \leq L_{\varrho,N} \abs{y -\tilde y} \qquad \text{for all $|y|, |\tilde y| \le N$.}
  \end{equation*}
\end{ass}
\begin{ass}
  \label{a:zeta}
  $\zeta(.,y) \in C^3(\R)$ for all $y\in \R$ and $\tfrac{\partial^{i}}{\partial x^i}\zeta(x,.)\in L^2(\R)$ for all $x\in \R$, $i\in\{0,1,2,3\}$. Moreover, 
  \begin{equation}
    \sup_{x\in \R} \norm{\tfrac{\partial^{i}}{\partial x^i} \zeta(x,.)}{L^2(\R)} <\infty,\quad i=0,1,...,3.\label{eq:Azeta}
  \end{equation}
  For the rest of this paper, we use the notation $\zeta^{(i)}:=\tfrac{\partial^{i}}{\partial x^i} \zeta$.
\end{ass}

\begin{example}[Convolution]
  Let $\zeta$ be a convolution kernel, i.\,e. $\zeta(x,y) := \zeta(x-y)$, $x$, $y\in \R$. If $\zeta\in C^{\infty}(\R)\cap H^3(\R)$, where $H^3$ denotes the Sobolev space of order $3$, 
  then Assumption~\ref{a:zeta} is satisfied. In this case, the operator $T_\zeta$ corresponds to spatial convolution with $\zeta$.
\end{example}

\begin{example}[Stochastic Stefan Problem]
  Let $\mu_+ =\mu_- \equiv 0$, $\sigma_+(x,v) = \sigma_-(-x,v) = v$ and $\varrho(x_1,x_2) = \varrho \* (x_2-x_1)$ for some $\varrho\in\R$. Then,~\eqref{eq:fbp} is the two-phase Stefan problem with multiplicative colored noise. With $\eta_2=0$, $\mu_-, \sigma_-\equiv 0$ and $\zeta(x,y):= \zeta(x-z)$ we end up with the one-phase system discussed in~\cite{sowersEtAl}. Even though our assumptions and proofs are formulated for the two-phase case, it is straight-forward to adapt them to a one-phase setting.
\end{example}

\begin{example}[Two-Phase Burger's equation]
  The case 
  \[\mu_+(x,v,v'):=\mu_-(-x,v,v'):= v\*v',\; x\in \R_{\geq 0},\; v,\,v'\in \R,\] 
  yields a stochastic version of a two-phase viscous Burger's equation in one dimension. Obviously, Assumption~\ref{a:mu} on $\mu_{\pm}$ is satisfied.
\end{example}

\begin{example}[Reaction-Diffusion-type drift]
  Set $\mu_\pm(x,v,v') := f_\pm(v)$, for some $f_\pm \in C^1(\R)$ with locally Lipschitz derivative and $f_+(0) = f_-(0) = 0$. Also in this case it is easy to check that Assumption~\ref{a:mu} is satisfied. 
\end{example}

Let us also remark here that without substantial change in our proofs the constant Laplacian terms $\eta_\pm \ddxx v$ in \eqref{eq:fbp} can be replaced by space-dependent Laplacians in the divergence form $\ddx \left(\eta_\pm(x - x_*(t)) \cdot \ddx v \right)$ for some scalar functions $\eta_\pm$ that are bounded by strictly positive constants.\\

Our first main result concerns the existence of a maximal local solution to the moving boundary problem \eqref{eq:fbp}.

\begin{thm}[Maximal Local solution]\label{thm:sfbp}
  Let Assumptions~\ref{a:mu},~\ref{a:sigma},~\ref{a:rho}, and~\ref{a:zeta} hold true and let $x_0 \in \R$ and $v_0 \in \Gamma(x_0)$. Then there exists a predictable, strictly positive stopping time $\tau$ and a local solution $(v,x_*)$ of~\eqref{eq:fbp} on the maximal interval $\llbrak 0, \tau \llbrak$ in the sense of Definition~\ref{df:fbp}. 
  For almost every $\omega \in \Omega$ it holds that $v(\omega;.) \in C([0, \tau(\omega)); H^1(\R))$ and $x_*(\omega,.) \in C^1([0,\tau(\omega)); \R)$. Moreover, $(v,x_*)$ is unique among all $H^1\oplus \R$-continuous solutions. 
\end{thm} 
\begin{rmk}
 The continuity statement implies that for all $x\in \R$, $t\mapsto v(t,x)$ is almost surely.
\end{rmk}

Imposing some additional assumptions on $\sigma_\pm$ and $\rho$ the solution becomes global.

\begin{ass}\label{a:mslg}
  The functions $b$ and $\tilde b$ in Assumption~\ref{a:mu} are (globally) bounded, and there exist functions $\sigma^1_\pm \in H^2(\R_+)\cap C^2(\R_{\geq 0})$ and $\sigma^2_\pm \in BUC^2(\R_{\geq 0})$, the space of all functions with bounded uniformly continuous second derivative. such that 
  \[ \sigma_+(x,y) = \sigma_+^1(x) + \sigma_+^2(x)y,\quad \sigma_-(-x,y) = \sigma_-^1(x) + \sigma_-^2(x)y\]
  for all $x \in \R_{\geq 0}$ and $y \in \R$.
\end{ass}

\begin{thm}[Global Solution]\label{thm:globalsol}
  If $\rho$ is bounded and Assumption~\ref{a:mslg} holds in addition to the assumptions of Theorem~\ref{thm:sfbp}, then $\tau = \infty$ almost surely, i.\,e. \eqref{eq:fbp} has a global solution.
\end{thm}

In Theorem~\ref{th:tau0eqtau} we provide a refined analysis of the case of finite-time blow-up ($\tau < \infty$). It turns out that under Assumption~\ref{a:mslg}, but with $\rho$ unbounded a finite-time blow-up of the system \eqref{eq:fbp} must coincide with a blow-up of the boundary terms.



\subsection{Overview of the proof}
Our treatment of equation \eqref{eq:fbp} consists of three steps
\begin{itemize}
\item Transformation into an equation with fixed boundary;
\item Formulation of the transformed equation as an abstract stochastic evolution equation;
\item Solving the abstract evolution equation by a fixed-point-argument.
\end{itemize}

For the first step we apply a change of coordinates 
\begin{equation}
  u_1(t,x) := v(t,x_*(t) + x),\quad\text{and}\quad u_2(t,x) := v(t,x_*(t) -x), \label{eq:trafo}
\end{equation}
i.e. new coordinates are defined relative to the free boundary $x_*(t)$, which yields
\begin{gather}
  \begin{split}    
    \d u_1(t, x)& = \left[\eta_1 \ddxx u_1+ \mu_+\left(x, u_1, \ddx u_1\right) \right. \\
    &\;\left. \vphantom{\ddxx}+ \frac{dx_*}{dt}(t) \cdot \left(\ddx u_1\right)\right] \d t  + \sigma_+\left(x, u_1\right)\d \xi_t (x_*(t) + x),\\
    \d u_2(t, x) &= \left[\eta_2 \ddxx u_2 + \mu_-\left(-x, u_2,\ddx  u_2 \right)\right.\\
    & \;\left. \vphantom{\ddxx} - \frac{dx_*}{dt}(t) \cdot \left(\ddx u_2\right)\right] \d t +  \sigma_-\left(-x, u_2\right)\d \xi_t (x_*(t) - x),\\
    \d x_*(t) &= \varrho \left( \ddx u_1(t,0+), \ddx u_2(t,0+)\right)\d t,
  \end{split}
  \label{eq:fbpt}
\end{gather}
for $t\geq 0$ and $x> 0$ with Dirichlet boundary conditions,
\begin{equation}
  \label{eq:bcD}\tag{D}
  u_1(t,0) = 0,\qquad u_2(t,0) = 0.
\end{equation}
Note that the classic chain rule is not sufficient to derive \eqref{eq:fbpt} from \eqref{eq:fbp}, since $v$ is not differentiable in time. Rather a special case of Ito's formula (a `stochastic chain rule') is needed to justify the computation. The transformation turns the moving boundary into a fixed boundary at $x=0$, but introduces an additional non-linear and unbounded drift term involving the spatial derivatives $\ddx u_1$ and $\ddx u_2$.\\

The second step is the abstract formulation of \eqref{eq:fbpt} in terms of the stochastic evolution equation 
\begin{equation}
  \d X(t) = \left[ \cA X(t) + \cB(X(t))\right] \d t + \cC(X(t)) \d W_t,\qquad X(0) = X_0. \label{eq:SEEq}
\end{equation}
where
\[ X(t) := (u_1(t,.), u_2(t,.), x_*(t)),\;t\geq 0\]
and $W$ is a cylindrical Wiener process with covariance operator $\Id$ on the separable Hilbert space $U=L^2(\R)$. Introducing the shorthand
\begin{equation}\label{eq:I}
\cI(u) = \left(\ddx u_1(0+), \ddx u_2(0+)\right)
\end{equation}
for the boundary terms, the coefficients of \eqref{eq:SEEq} are given by
\begin{align}
  \cA =& \begin{pmatrix}
    \eta_+ \Delta & 0 & 0 \\ 
    0 & \eta_- \Delta  & 0 \\
    0 & 0 &0
  \end{pmatrix} - c\; \textrm{id}, \label{eq:coefA}\\
  \cB(u)(x) =& \begin{pmatrix}
    \mu_+(x, u_1(x), \ddx u_1(x))  + \ddx u_1(x) \* \varrho \left(\cI(u(t))\right) \\
    \mu_-(-x, u_2(x), \ddx u_2(x)) - \ddx u_2(x) \* \varrho \left(\cI(u(t))\right) \\
    \varrho \left(\cI(u(t))\right) 
  \end{pmatrix} + c\; \textrm{id},\label{eq:coefB}\\
  \cC(u)(w)(x) =& \begin{pmatrix}
    \sigma_+(x, u_1(x)) T_\zeta w (u_3+x)  \\
    \sigma_-(-x, u_2(x)) T_\zeta w(u_3-x)\\
    0
  \end{pmatrix},\quad u\in \DA,\,w\in U,\,x\geq 0.\label{eq:coefC}
\end{align} 
Here, $\Delta$ is the Laplacian on $\R_+$ with Dirichlet boundary conditions and $c > 0$ is an arbitrary constant, whose sole function is to move the spectrum of $\cA$ into the negative half-line $(-\infty,0)$.\\

Finally, the solution of \eqref{eq:fbp} in Theorem~\ref{thm:sfbp} will be obtained from the unique strong solution $X$ on $\llbrak 0, \tau \llbrak$ of the stochastic evolution equation~\eqref{eq:SEEq} with initial data $X_0:= (v_0(.+x_0)\vert_{\R_+}, v_0(x_0-.)\vert_{\R_+}, x_0)$ by setting $X = (u_1,u_2,x_*)$ and 
\[ v(t,x) :=
\begin{cases}
  u_1(t, x-x_*(t)), & x>x_*(t), \\
  u_2(t, x_*(t)-x), & x<x_*(t),
\end{cases}
\]
In the remainder of the paper we will make the above steps rigorous, by traversing them in the reverse direction:
\begin{itemize}
\item In section \ref{sec:seeqs} we show that under certain assumptions the abstract stochastic evolution equation \eqref{eq:SEEq} has a unique strong solution.
\item In section \ref{sec:proofs} we show that the parameter assumptions made in Section~\ref{Sub:Assumptions} are sufficient for the assumptions of section~\ref{sec:seeqs}
\item In section \ref{sec:trafo} we show the stochastic chain rule that is necessary to make the transformation to fixed boundary rigorous and collect all pieces to complete the proof of our main results.
\end{itemize}
\begin{rmk}\label{rmk:generalbc}
Considering carefully our proof of the existence result in the next section it can be seen that equation~\eqref{eq:SEEq} can be solved also for homogeneous Neumann or even Robin boundary conditions. Of course the boundary conditions on $\sigma_\pm$ in Assumption~\ref{a:sigma}~\ref{ai:sigmabc} have to be adapted accordingly. 
  The main difference to the case of Dirichlet boundary conditions, is that a discontinuity at the boundary introduces a jump to the dynamics of $v$ in equation \eqref{eq:fbp} at any given point $x \in \RR$, every time the boundary $x_*(t)$ crosses $x$. In particular the `stochastic chain rule' developed in Section~\ref{sec:trafo} is no longer sufficient to pass from the moving boundary equation \eqref{eq:fbp} to the fixed boundary equation \eqref{eq:fbpt} and back. 
\end{rmk}
\begin{rmk}
  The existence result for the centered equations~\eqref{eq:SEEq} can be extended to the case with Brownian noise in the boundary, without any problems. That is,
\[\d x_*(t) =  D(u(t))\d t +\sigma_* \d B_t\]
where $D$ is any locally Lipschitz operator from $\DA$ into $\R$, and $\sigma_*>0$. Here, $B$ can be either independent of $W$, or a Hilbert-Schmidt transformation of $W$ into $\R$. 
\end{rmk}


%% file: stocheveqDA.tex
\subsection{Preliminaries}
In this section we concentrate on the evolution equation \eqref{eq:SEEq}, i.e. 
\begin{equation}
  \d X(t) = [AX(t) + B(X(t))] \d t  + C(X(t)) \d W_t,\quad t\geq 0,\label{eq:seebd}
\end{equation}
where $W$ is a cylindrical Wiener process with covariance operator $\Id$ on a separable Hilbert space $U$. At this point it is sufficient to assume that $X$ takes values in an arbitrary separable Hilbert space $E$ with norm $\norm{.}{}$. On the coefficients $A,B,C$ we will impose assumptions that are milder (but also more abstract) than the assumptions made in section~\ref{sec:sfbp} on the coefficients of the free boundary problem. As will be shown in section~\ref{sec:proofs} the assumptions below are implied by the assumptions from section~\ref{Sub:Assumptions} such that eventually the results on the evolution equation \eqref{eq:SEEq} can be used to solve the free boundary problem~\eqref{eq:fbp}. Nevertheless, the results of this section may be of independent interest when generalizations of \eqref{eq:fbp} are considered.\\

 On the operator $A$ in \eqref{eq:seebd} we make the following assumption.
\begin{ass}
  \label{a:A}
  $A$ is a densely defined and sectorial operator with domain $\dom(A)\subset E$. Moreover, the resolvent set of $A$ contains $[0,\infty)$ and there exists a $M > 0$ such that the resolvent $R(\lambda,A)$ satisfies 
  \begin{equation}\label{eq:resolvent_eq}
    \norm{R(\lambda,A)}{} \le \frac{M}{1 + \lambda}, \qquad \text{for all $\lambda > 0$.}
  \end{equation}
\end{ass}
\begin{rmk}
  This assumption is equivalent to each of the following statements
  \begin{itemize}
  \item  Equation \eqref{eq:resolvent_eq} holds and the resolvent set of $A$ contains $0$ and a sector
    \[\{\lambda \in \C: |\arg \lambda | < \theta\}\]
    for some $\theta \in (\pi/2,\pi)$. 
  \item The operator $A$ is sectorial and $-A$ is positive in the sense of \cite{lunardiInterpol}.
  \end{itemize}
  Assumption \eqref{a:A} ensures that $A$ generates an analytic semigroup $(S_t)_{t \geq 0}$ and that suitable interpolation spaces between $E$ and $\dom(A)$ can be defined through fractional powers of $-A$. They also imply that the semigroup $S_t$ is of strictly negative type, i.e. there exist $\delta$, $M > 0$ such that $\norm{S_t}{} \le M e^{-\delta t}$. Note, that if $M=1$ then $S_t$ is a contraction semigroup, which we shall not assume a priori.
\end{rmk}

Using the semigroup $S_t$ that is generated by $A$ we can introduce the important concept of mild solutions.

\begin{defn}
  Let $X= (X(t))$ be a $\dom(A)$-valued predictable process and $\tau$ be a predictable stopping time. 
  \begin{itemize}
  \item $X$ is called global mild solution to the stochastic evolution equation~\eqref{eq:seebd} on $\dom(A)$ with initial data $X_0\in \dom(A)$, if  \begin{equation}
      X(t) = S_tX_0 + \int_0^t S_{t-s} B(X(s)) \d s + \int_0^t S_{t-s}C(X(s)) \d W_s.\label{eq:seebdm}
    \end{equation}
    holds for all $t \geq 0$, $\PP$-a.s.
  \item $X$ is called mild solution on $\llbrak 0, \tau \llbrak$, if~\eqref{eq:seebd} holds on the stochastic interval $\llbrak 0, \tau \llbrak$.
  \item The stochastic interval $\llbrak 0, \tau \llbrak$ is called maximal for $X$ if there is no $\dom(A)$-continuous extension of $X$ to a larger stochastic interval.
  \end{itemize}
  In the last two terms of \eqref{eq:seebdm}, $\int$ denotes the Bochner and stochastic integral on the Hilbert space $\dom(A)$, respectively. If we want to emphasize the underlying space $\dom(A)$ we write global mild $\dom(A)$-solution and mild $\dom(A)$-solution respectively.
\end{defn}

Finally we will be able to show that the mild solution is also a strong one in the following sense:

\begin{defn}\label{df:strongsolution}
  Given $\DA$-valued initial data $X_0$ and a predictable stopping time $\tau$, $X$ is called strong solution of~\eqref{eq:SEEq} on $\llbrak 0, \tau \llbrak$, if $X$ is a $\DA$-valued predictable process and  
  \begin{equation}
    X(t) = X_0 + \int_0^{t }A X(s) + B(X(s))\d s + \int_0^{t} C(X(s)) \d W_s, 
  \end{equation}
  holds on $\llbrak 0, \tau \llbrak$. Global solutions and maximality are defined in the same way as for mild solutions.
\end{defn}



\subsection{Interpolation spaces}
  By taking fractional powers of $-A$ we introduce inter- and extrapolation spaces for $E$. For $\alpha > 0$ we define
  \begin{equation}
    \label{eq:25}
    E_\alpha := \dom((-A)^\alpha),\qquad \norm{h}{\alpha} := \norm{(-A)^\alpha h}{E}, \; h\in E_\alpha.
  \end{equation}
  It is known that also $E_\alpha$ with the induced scalar product is a separable Hilbert space. In particular, $\norm{.}{1}$ is equivalent to the graph norm of $A$ and the following continuous embedding relations hold for $\alpha \in [0,1]$:
  \begin{equation}
    \label{eq:15}
    \dom(A) = E_1 \hookrightarrow E_\alpha \hookrightarrow E_0 = E, 
  \end{equation}

  Note that the restriction of $A$ to any $E_\alpha, \alpha \in [0,1]$ is again a densely defined and closed operator on $E_\alpha$. Moreover, it is the infinitesimal generator of the restriction of $S_t$ to $E_\alpha$, which is again an analytic (contraction) semigroup; see e.g. \cite[Ch.~II.5]{Engel}.

The following regularity property of $S_t$ between different interpolation spaces $E_\alpha$, $\alpha \in [0,1]$ will be crucial in the proofs that follow. We derive it from results in~\cite{lunardiInterpol} on interpolation spaces.
\begin{lem}\label{lem:StHa}
  Let $\beta \geq 0$ and $\alpha > \beta$. Then, for all $t>0$ and $h\in E_{\beta}$,
  \[ \norm{S_t h}{\alpha} \leq K_{\alpha,\beta} \frac{1}{t^{\alpha - \beta}} e^{-\delta t}\norm{h}{\beta}. \]
\end{lem}
Note that the factor in front of $\norm{h}{\beta}$ is integrable at time $t = 0$, which is the key property used in the estimates concerning the mild formulation of~\eqref{eq:seebd} on $E_1$.
\begin{proof}
  Suppose first that $\alpha = \beta + n$ for some $n\in \N$, then we get from~\cite[Thm~1.5.2d and p. 70]{pazySemigroups} that there exists $K_n>0$ such that 
  \begin{equation}
    \label{eq:abn}
    \norm{S_t h}{\alpha} = \norm{(-A)^nS_t (-A)^\beta h}{0} \leq K_n t^{-n} e^{-\delta t} \norm{h}{\beta}.  
  \end{equation}
  Now assume that $\alpha \in (\beta + n, \beta + n +1)$ for some $n\in \N_0$ and set $\theta = \alpha - \beta - n\in (0,1)$. By \cite[Prop. 4.7]{lunardiInterpol} the real interpolation space $(E_0,E_1)_{\theta, 1}$ is continuously embedded into $\dom((-A)^\theta)$. Combining this fact with \cite[Cor. 1.7]{lunardiInterpol} we obtain that there exists $K > 0$ such that 
  \[\norm{h'}{\theta} \le K \norm{h'}{0}^{1 - \theta} \norm{h'}{1}^{\theta} \quad \text{for all $h' \in E_1 = \dom(A)$}.\]
  Now let $h \in \dom((-A)^\beta)$ and set $h' = (-A)^n S_t (-A)^\beta h \in \dom(A)$. Applying the above inequality and using boundedness of the semigroup $S_t$ we obtain
  \[\norm{S_t h}{\alpha}  = \norm{h'}{\theta} \le K' \norm{(-A)^nS_t (-A)^\beta h}{0}^{1 - \theta} \norm{(-A)^nS_t (-A)^\beta  h}{1}^\theta.\]
  Finally,~\eqref{eq:abn} for $n$ and $n+1$ yields
  \[ \norm{S_t h}{\alpha} \le K'' e^{-\delta t} t^{-(1-\theta)n - \theta (n+1)} \norm{h}{\beta},\]
  proving the result. 
\end{proof}

To deal with the singularity in $0$ on the right hand side above, we will use an extended version of Gronwall's lemma, see~\cite[Lem 7.0.3]{lunardiAnalytic} or, for a proof,~\cite[p. 188]{henryGeo}.
\begin{lem}[Extended Gronwall's lemma]\label{lem:gronwall}
  Let $\alpha>0$, $a$, $b\geq 0$, $T \geq 0$, and $u:[0,T]\to\R$ be non-negative and integrable. If, for all $t\in [0,T]$,
  \begin{equation}
    \label{eq:6}
    u(t) \leq a + b\int_0^t u(s) (t-s)^{\alpha-1} \d s,
  \end{equation}
  then exists a constant $K_{\alpha, b,T}$, depending only on $\alpha$, $b$ and $T$, such that,
  \begin{equation}
    \label{eq:18}
    u(t) \leq a K_{\alpha,b,T},\qquad t\in [0,T].
  \end{equation}
\end{lem}

\subsection{Existence of global mild solutions}
We start by discussing global solutions. Subsequently, the existence of local solutions under milder assumptions will be shown by localizing with appropriate stopping times. 
Denoting by $\HS(U,E_1)$ the (Hilbert) space of Hilbert-Schmidt operators from $U$ to $E_1$ we introduce the following Lipschitz-type assumption, which will imply the existence of global mild solutions to \eqref{eq:seebd}.
\begin{ass}
  \label{a:SEEglobal}
  There exists $\alpha \in (0,1]$ such that $B: E_1 \rightarrow E_\alpha$ and $C:E_1\to \HS(U,E_1)$ are Lipschitz continuous, i.e. there exists a constant $\hat L$ such that 
  \begin{equation}
    \label{eq:16}
    \norm{B(Y) - B(Z)}{\alpha} + \norm{C(Y)-C(Z)}{\HS(U, E_1)} \leq \hat L\norm{Y- Z}{1},
  \end{equation}  
  holds for all $Y$, $Z\in E_1$.
\end{ass}
\begin{rmk}
  Assumption~\ref{a:SEEglobal} implies a linear growth bound on $B$ and $C$ in the sense that 
  \begin{equation}
    \label{eq:LG} 
    \norm{B(Y)}{\alpha} + \norm{C(Y)}{\HS(U, E_1)} \leq \hat M (1+\norm{Y}{1}).
  \end{equation}
  for all $Y\in E_1$.
\end{rmk}

\begin{thm}[Global Mild Solution of \eqref{eq:seebd}]\label{thm:EUC}
  Let Assumption~\ref{a:A} and~\ref{a:SEEglobal} hold true and let $p>1$. Then, for every initial data $X_0 \in L^{2p}(\Omega; E_1)$ there exists a unique global mild solution $X$ of \eqref{eq:seebd} on $E_1$. Moreover,
  \begin{equation}
    \E{\sup_{0\leq t\leq T} \norm{X(t)}{1}^{2p}} \leq \hat K \left(1 + \E{\norm{X_0}{1}^{2p}}\right),\label{eq:XX02p}
  \end{equation}
  for all $T\geq 0$ and $X$ is $E_1$-continuous almost surely. If, in addition, $S_t$ is a contraction semigroup, then the statement is true even for $p=1$.
\end{thm}
\begin{rmk}
  Without much effort one can extend the theorem to time- and path-dependent predictable coefficients $B: \Omega\times \R_{\geq 0} \times E_1 \to E_\alpha$ and $C:\Omega\times \R_{\geq 0} \times E_1\to \HS(U,E_1)$, provided that~\eqref{eq:16} and~\eqref{eq:LG} hold. The same is true for Theorem~\ref{thm:EUCloc} below.
\end{rmk}

\begin{proof}
The theorem will be shown using a fixed-point argument. Using Lemma~\ref{lem:StHa} we will be able to prove that the following mapping is a contraction. We define
  \[ \Kfix(Y)(t) := S_tX_0 + \int_0^t S_{t-s}B(Y(s)) \d s + \int_0^t S_{t-s}C(Y(s)) \d W_s,\]
  for elements $Y$ out of the Banach space
  \[ \HTp := \setc{ Y:\Omega \times [0,T] \rightarrow E_1, \text{ predictable}}{ \normTp{Y}  <\infty},\quad T>0,\]
  equipped with the norm defined by
  \[\normTp{Y}^{2p} := \sup_{0\leq t\leq T} \E{\norm{Y(t)}{1}^{2p}}.\]
  To show the contraction property of $\Kfix$ on $\HTp$ for small enough $T>0$, we first decompose
  \[\Kfix(Y)(t) =: S_t X_0 + \Kfix_B(Y)(t) + \Kfix_C(Y)(t),\quad Y\in \HTp,\] 
  where $\Kfix_B$ is the convolution of $S$ with $B$ and $\Kfix_C$ is the stochastic convolution with $C$, respectively. 
  The first term is easiest to handle. From the strong continuity and boundedness of $S_t$ we get
  \begin{equation}
    \normTp{S_{(.)}X_0}^{2p} = \sup_{0\leq t\leq T} \E{\norm{S_tX_0}{1}^{2p}}  \leq  K_{1,1} \E{\norm{X_0}{1}^{2p}},
  \end{equation}
  where the constant $K_{1,1}$ depends on the bound for the norm of the semigroup. The term $\Kfix_B(Y)$ is more difficult to handle. Let $Y\in \HTp$, then by Bochner's inequality, Lemma~\ref{lem:StHa}, Jensen's inequality and the growth estimate~\eqref{eq:LG}
  \begin{gather}
    \begin{split}
      \normTp{\Kfix_B(Y)}^{2p} &\leq \sup_{0\leq t\leq T}\E{\left(\int_0^t \norm{S_{t-s}B(Y(s))}{1}\d s\right)^{2p}} \\
      &\leq K \sup_{0\leq t\leq T}\E{\left(\int_0^t\frac{\norm{B(Y(s))}{\alpha}}{(t-s)^{{1-\alpha}}} \d s\right)^{2p}}\\
     &\leq K T^{\alpha (1-2 p)} \sup_{0\leq t\leq T}\E{\int_0^t\frac{\left(1 + \norm{Y(s)}{1}\right)^{2p}}{(t-s)^{{1-\alpha}}} \d s} \\
      &\leq K  T^{2p\alpha} \left(1+ 
        \alpha T^{-\alpha}\sup_{0\leq t\leq T} \E{ \int_0^t \norm{Y(s)}{1}^{2p} \frac{\d s}{(t-s)^{{1-\alpha}}}}\right),
    \end{split}\label{eq:KBim}
  \end{gather}
  with constant $K$ changing from line to line, but depending only on $p$, $\alpha$ and $\hat M$. Note that to apply Jensen's inequality we have used that $(t-s)^{\alpha-1}\d s$ is a finite measure on $(0,t)$ with mass $\sfrac{t^{\alpha}}{\alpha}$, and that the inequality $\norm{a+b}{}^{2p} \le 2^p \left(\norm{a}{}^{2p} + \norm{b}{}^{2p}\right)$  has entered in the last step. Applying the Fubini-Tonelli theorem yields
  \[\sup_{0\leq t\leq T} \E{ \int_0^t \norm{Y(s)}{1}^{2p} \frac{\d s}{(t-s)^{{1-\alpha}}}} \leq \frac1{\alpha}T^{\alpha}\sup_{0\leq t\leq T} \E{\norm{Y(s)}{1}^{2p}}.\]
  Inserting into \eqref{eq:KBim} we get that 
  \[\normTp{\Kfix_B(Y)}^{2p} \leq K T^{2p\alpha} \left(1+\normTp{Y}^{2p}\right) < \infty.\]
  Let now $Y$, $Z\in \HTp$, then with the same arguments as in~\eqref{eq:KBim}, but with the Lipschitz estimate \eqref{eq:16} instead of \eqref{eq:LG} we obtain
  \begin{multline*} 
    \normTp{\Kfix_B(Y) - \Kfix_B(Z)}^{2p}\leq  K_{1,\alpha}^{2p}\sup_{0\leq t\leq T}\E{\left(\int_0^t\frac{\norm{B(Y(s))- B(Z(s))}{\alpha}}{(t-s)^{{1-\alpha}}} \d s\right)^{2p}}\\
    \leq K'T^{(2p-1)\alpha}\sup_{0\leq t\leq T}\E{\int_0^t\norm{Y(s)- Z(s)}{1}^{2p}\frac{\d s}{(t-s)^{{1-\alpha}}}}.
  \end{multline*}
  Applying again the Fubini-Tonelli theorem yields
  \[\normTp{\Kfix_B(Y) - \Kfix_B(Z)}^{2p}\leq K'' T^{2p\alpha} \normTp{Y - Z}^{2p}.\]
  Here, the constants $K'$ and $K''$ depend only on $p$, $\alpha$ and $\hat L$. 

  To show similar properties for the stochastic convolution $\Kfix_C$ is exactly the same as in the proof of the classical result~\cite[Theorem 7.2, see pp.189]{dPZinf}. 
  Everything together yields constants $K_B$ and $K_C$ independent of $X_0$, such that
  \begin{gather*}
    \begin{split}
      \normTp{\Kfix(Y) - \Kfix(Z)} &\leq   \normTp{\Kfix_B(Y) - \Kfix_B(Z)} + \normTp{\Kfix_C(Y) - \Kfix_C(Z)}\\
      &\leq  (K_B+K_C) T^{\alpha} \normTp{Y-Z}<1,
    \end{split}
  \end{gather*} 
  provided $T< (K_B + K_C)^{-\alpha}$. Hence, $\Kfix$ is a contraction on $\HTp$ and possesses a unique fixed point, which is a mild solution of \eqref{eq:seebd} up to time $T > 0$. Concatenating solutions, we obtain a global solution. Finally, to show the the uniqueness claim, we consider two arbitrary solutions $X_1$ and $X_2$ and the stopping times
  \[\tau_R := \inf\left\{t\geq 0 \;\vert\; \norm{\Kfix_B(X_i)}{1} \geq R,\text{ or }\norm{\Kfix_C(X_i)}{1} \geq R,\; i\in \{1,2\}\right\}.\]
Using the standard procedure as in the proof of~\cite[Theorem 7.2]{dPZinf}, but using the estimates for $\Kfix_B$ from above and Lemma~\ref{lem:gronwall} we obtain that the solutions $X_1$ and $X_2$ must coincide up to the stopping time $\tau_R$. Passing to the limit $R\to \infty$, global uniqueness follows. 
    The remaining part, namely showing~\eqref{eq:XX02p} and the continuity claim, is subject of Lemma~\ref{lem:EsupX} and~\ref{lem:cont} below.
\end{proof}

\begin{lem}\label{lem:EsupX}
  Let Assumption~\ref{a:A} and \eqref{eq:LG} hold true and let $p > 1$. Let $X$ be a mild solution on $[0,T]$ of~\eqref{eq:seebd} with initial value $X_0 \in L^{2p} (\Omega; E_1)$ such that 
  \[ \E{\int_0^T\norm{X(t)}{1}^{2p} \d t} < \infty.\]
  Then, 
  \[ \E{\sup_{0\leq t\leq T} \norm{X(t)}{1}^{2p}} \leq K_{p,T} \left(1 + \E{\norm{X_0}{1}^{2p}}\right).\]
  If, in addition, $S_t$ is a contraction semigroup, then the statement is true even for $p = 1$.
\end{lem}
\begin{rmk}
We emphasize that the Lipschitz property \eqref{eq:16} is not needed to show this Lemma.
\end{rmk}

\begin{proof}
  We use the notation from the previous proof and write
  \[ X_t = S_t X_0 + \Kfix_B(X)_t  + \Kfix_C(X)_t.\]
  First, note that the integrability assumption on $X$ and the linear growth property~\eqref{eq:LG} yield
  \begin{equation}
    \label{eq:EintBCfinite}
    \EE \int_0^T \norm{B(X(t))}{\alpha}^{2p} + \norm{C(X(s)}{\HS(U;E_1)}^{2p}.
  \end{equation}
  For the case $p=1$ we may assume that $S_t$ is a contraction semigroup. Hence, we can apply~\cite[Theorem 6.10]{dPZinf} which gives
  \begin{equation}
    \E{\sup_{0\leq t\leq T} \norm{\Kfix_C(X)(t)}{1}^{2}} \leq K_1 \E{\int_0^T\norm{C(X(s))}{\HS(U,E_1)}^2\d s}. \label{eq:EsupW2}
  \end{equation}
  For the case $p>1$ we use that $S_t$ is a $C_0$-semigroup and apply~\cite[Theorem 1.1]{dPZstochConv} which yields
  \begin{equation}
    \E{\sup_{0\leq t\leq T} \norm{\Kfix_C(X)(t)}{1}^{2p}} \leq K_p T^{p-1} \E{\int_0^T\norm{C(X(s))}{\HS(U,E_1)}^{2p}\d s} <\infty. \label{eq:EsupWp}
  \end{equation}
  In both cases the growth bound~\eqref{eq:LG} yields
  \begin{equation}\label{eq:EsupWp2}
\E{\int_0^T\norm{C(X(s))}{\HS(U,E_1)}^{2p}\d s} \leq 2^{2p} {\hat M}^{2p}\left(T^{2p} + \E{\int_0^T\norm{X(s)}{1}^{2p} \d s}\right).
\end{equation}
  
  For the drift part we again use the linear growth bound~\eqref{eq:LG}, and proceeding similar to~\eqref{eq:KBim} we obtain
  \begin{multline*}
    \sup_{0\leq s\leq t} \norm{X(s)}{1}^{2p}\\
    \leq K_{\alpha, p,T} \left(1 + \norm{X_0}{1}^{2p} + \int_0^t \sup_{0\leq r\leq s}\norm{X(r)}{1}^{2p}\tfrac{\d s}{(t-s)^{{1-\alpha}}} + \sup_{0\leq s\leq t}\norm{\Kfix_C(s)}{1}^{2p} \right),
  \end{multline*}
  for $t\leq T$. Taking expectations, using the Fubini-Tonelli theorem and inserting the estimates concerning $\Kfix_C$ yields
  \begin{multline*} 
    \E{\sup_{0\leq s\leq t} \norm{X(s)}{1}^{2p}}   \leq K_{\alpha,p,T} \left(1 +\E{ \norm{X_0}{1}^{2p}}\right.\\
    \left. + \int_0^t \E{\sup_{0\leq r\leq s}\norm{X(r)}{1}^{2p}}\tfrac{\d s}{(t-s)^{{1-\alpha}}} + \int_0^t\E{\sup_{0\leq s\leq t}\norm{X(s)}{1}^{2p}} \d s \right)\\
    \leq  K_{\alpha,p,T} \left(1 +\E{ \norm{X_0}{1}^{2p}} + \int_0^t \E{\sup_{0\leq r\leq s}\norm{X(r)}{1}^{2p}}\tfrac{\d s}{(t-s)^{{1-\alpha}}}\right).
  \end{multline*}
  Finally, Gronwall's lemma~\ref{lem:gronwall} finishes the proof.
\end{proof}
$E_1$-continuity of the stochastic convolution $\Kfix_C(X)_t$ follows from standard results and the estimate \eqref{eq:EsupWp2}. However, for the $E_1$-continuity of $\Kfix_B(X)$ we provide a detailed proof, since we can in general not assume that $B$ is $E_1$-valued. To this end, we modify slightly the result~\cite[Proposition 4.2.1]{lunardiAnalytic} and its proof. 

\begin{lem}
  \label{prop:cdrift}
  Let $\psi: [0,T]\rightarrow E_\alpha$ be integrable and such that 
  \[\sup_{0\leq t\leq T} \norm{\psi(s)}{\alpha} =: \bar \psi <\infty.\]
  Then, $(S*\psi)_t := \int_0^t S_{t-s} \psi(s) ds$ is in $C([0,T], E_1)$.
\end{lem}
\begin{proof}
  Note that for $0<t\leq T$ and arbitrary $0<\epsilon<t$,
  \[ \tfrac{\d}{\d t}S_t\varphi =  \tfrac{\d}{\d t} S_{t-\epsilon}S_\epsilon \varphi = AS_{t-\epsilon} S_\epsilon \varphi= AS_t \varphi,\quad\forall \varphi \in E,\]
  where we use that $S_\epsilon\varphi \in \dom(A)$ and $ \tfrac{\d}{\d t} S_t = AS_t$ on $\dom(A)$; see e.g. \cite[Lem. II.1.3]{Engel}. We now observe that for $0<s<t\leq T$
  \begin{gather*}
      \int_0^t S_{t-r} \psi(r)\d r - \int_0^s S_{s-r}\psi(r)\d r 
      = \int_0^s \int_{s-r}^{t-r} AS_{u} \psi(u)\d u\d r + \int_s^t S_{t-r}\psi(r) \d r.
  \end{gather*}
  Hence, with Bochner's inequality and Lemma~\ref{lem:StHa}
  \begin{multline*}
    \norm{(S*\psi)_t - (S*\psi)_s}{1}\\
    \leq K_{2,\alpha}\int_0^s \int_{s-r}^{t-r}  u^{\alpha-2}\norm{\psi(u)}{\alpha} \d u\d r + K_{1,\alpha} \int_s^t u^{\alpha-1} \norm{\psi(u)}{\alpha} \d u\\
    \leq  \bar\psi  K_{\alpha} \left(\abs{t-s}^{\alpha}+\abs{t^{\alpha} - s^{\alpha}}\right) 
    \longrightarrow 0,\quad\text{as}\quad\abs{t-s}\rightarrow 0.
  \end{multline*}
  For $s=0$ we get directly with Bochner inequality and Lemma~\ref{lem:StHa} for $t\searrow 0$,
  \begin{equation*}
    \norm{\int_0^t S_{t-r}\psi(r)\d r}{1}
    \leq K_{1,\alpha} \bar{\psi}  \int_0^t (t-r)^{\alpha-1}\d r
    = \frac1\alpha \bar{\psi} K_{1,\alpha} t^{\alpha} \rightarrow 0. \qedhere
  \end{equation*}
\end{proof}


\begin{lem}\label{lem:cont}
  Under the assumptions of Lemma~\ref{lem:EsupX} the mild solution $X$ of~\eqref{eq:seebd} is almost surely $E_1$-continuous.
\end{lem}
\begin{proof}
  As above, we decompose
  \[X_t = S_t X_0 + \Kfix_B(X)_t + \Kfix_C(X)_t.\]
  Continuity of the first summand in the decomposition is immediate, since $S_t$ is strongly continuous. From Lemma~\ref{lem:EsupX} we get that $\psi(t) := B(X)_t$ satisfies the conditions of Lemma~\ref{prop:cdrift} and hence continuity of $\Kfix_B(X)_t$ follows. In the case $p>1$, estimate \eqref{eq:EsupWp2} together with \cite[Theorem 1.1]{dPZstochConv} yields continuity of $\Kfix_C(X)$. For $p=1$ we may assume that $S_t$ is a contraction semigroup on $E_1$ and apply~\cite[Theorem 6.10]{dPZinf} instead. Note that we are always using the continuous modifications of the stochastic integrals\slash{}convolutions.
\end{proof}

Together, these Lemmas complete the proof of Theorem~\ref{thm:EUC}.

\subsection{Existence of local mild solutions}
To obtain only local solutions up to a stopping time $\tau$, we can relax the assumptions on $B$ and $C$ made in the previous subsection.

\begin{ass}
  \label{a:SEElocal}
  There exists $\alpha \in (0,1]$ such that $B: E_1 \rightarrow E_\alpha$ and $C: E_1\to \HS(U,E_1)$ are Lipschitz continuous on bounded sets, i.\,e. for all $N\in \N$ there exists $L_N$ such that   \begin{equation}
    \label{eq:BDlip}
    \norm{B(Y) - B(Z)}{\alpha} + \norm{C(Y)-C(Z)}{\HS(U, E_1)} \leq L_N\norm{Y- Z}{1}
  \end{equation}
  holds for all $Y$, $Z\in E_1$ with $\norm{Y}{A}$, $\norm{Z}{A} \leq (N + 1)$.
\end{ass}

\begin{rmk}
  Assumption~\ref{a:SEElocal} yields for all $N\in \N$ constants $M_N$ such that for all $Y\in E_1$ with $\norm{Y}{1}\leq (N + 1)$
  \begin{equation}
    \label{eq:BDlg}
    \norm{B(Y)}{\alpha} + \norm{C(Y)}{\HS(U, E_1)} \leq M_N \left(1 + \norm{Y}{1}\right).
  \end{equation}
\end{rmk}

\begin{thm}[Local Mild Solution of \eqref{eq:seebd}]\label{thm:EUCloc}
  Let Assumption~\ref{a:A} and~\ref{a:SEElocal} hold true and let $p >1$. Then, for every initial data $X_0 \in L^{2p}(\Omega; E_1)$ there exists a unique mild $E_1$-solution $X$ of \eqref{eq:seebd} on a maximal stochastic interval $\llbrak 0, \tau \llbrak$. Moreover, $X$ is $E_1$-continuous on $\llbrak 0, \tau \llbrak$, $\tau > 0$ and $\lim_{t\nearrow \tau}\norm{X(t)}{1} = \infty$ on $\{\tau <\infty\}$ almost surely. 
 If, in addition, $S_t$ is a contraction semigroup, then the statement is true even for $p=1$.
\end{thm}

We use the following localization method, similar to the truncation in~\cite{sowersEtAl}. For each $N\in \N$ fix a monotone decreasing function $h_N \in C^\infty(\R_{\geq 0})$ with 
\begin{equation}
  \label{eq:truncfct}
  h_N(x) = \begin{cases} 1, & x\leq N, \\ 0,&x\geq N+1,\end{cases}
\end{equation}and for a constant $c>0$,
\[ \sup_{N\in \N}\norm{\ddx h_N}{\infty} \leq c.\]
Define the truncated coefficients
\[ B_N(u) := h_N(\norm{u}{1}) B(u),\quad C_N(u) :=h_N(\norm{u}{1}) C(u),\]
and consider the localized stochastic evolution equation
\begin{equation}
  \d X^{(N)}(t) = \left[A X^{(N)}(t) + B_N( X^{(N)}(t)) \right]\d t + C_N( X^{(N)}(t)) \d W_t.\label{eq:SEEqN}
\end{equation}
\begin{lem}
  \label{lem:BNCNglobal}
  Let $B$ and $C$ be such that the local Lipschitz assumption~\ref{a:SEElocal} holds. Then, $B_N$ and $C_N$ satisfy the global Lipschitz assumption~\ref{a:SEEglobal} for all $N\in\N$.
\end{lem}
\begin{proof}
  First, it is obvious to see that
  \begin{align}
    \label{eq:BNCNbd}
    \norm{B_N(Y)}{\alpha} + \norm{C_N(Y)}{\HS(U,E_1)} &\leq
    \begin{cases}
      M_{N+1}(2+N) \norm{Y}{1},& \norm{Y}{1} \leq N+1,\\
      0,& \norm{Y}{1} > N+1
    \end{cases}\\
    &\leq M_N \left(1 + \norm{Y}{1}\right). \notag
  \end{align}
  For the global Lipschitz continuity let $Y$, $Z\in E_1$ and assume, without loss of generality, that $\norm{Y}{1} \geq \norm{Z}{1}$. Then, write
  \begin{equation}
    \label{eq:20}
    B_N(Y) - B_N(Z) = h_N(\norm{Y}{1})\left(B(Y) - B(Z)\right)+ B(Z) \left(h_N(\norm{Y}{1}) - h_N(\norm{Z}{1})\right).
  \end{equation}
  If $\norm{Y}{1}> N+1$, then the first term vanishes. Else, it holds that $\norm{Z}{1}\leq \norm{Y}{1} \le N+1$ and thus, in both cases,
  \begin{equation}
    \label{eq:21}
    \norm{h_N(\norm{Y}{1})(B(Y) - B(Z))}{\alpha} \leq L_N \norm{Y-Z}{1}.
  \end{equation}
  The second term in \eqref{eq:20} vanishes if $\norm{Z}{1} > N+1$. Otherwise,
  \begin{equation}
    \label{eq:22}
    \norm{B(Z)}{\alpha} \abs{h_N(\norm{Y}{1}) - h_N(\norm{Z}{1})} \leq 2 c M_N(2 + N) \norm{Y-Z}{1},
  \end{equation}
  where we applied chain rule and mean value theorem for Fr\'echet derivatives. Of course, replacing $B$ by $C$ and $\norm{.}{\alpha}$ by $\norm{.}{\HS(U,E_1)}$ changes nothing in the computation so that we get a global Lipschitz constant $\hat L$, depending on $c$, $M_N$ and $L_N$.
\end{proof}

For the proof of theorem~\ref{thm:EUCloc} we may assume $X_0\in L^{2p}(\Omega:E_1)$, some $p\geq 1$, to be given and Assumptions~\ref{a:A} and~\ref{a:SEElocal} to be true. For $N\in\N$ we then denote by $X^{(N)}$ the unique mild solution to the localized equation~\eqref{eq:SEEqN}, which exists due to Theorem~\ref{thm:EUC}. To relax the truncation, we introduce the stopping times
\begin{equation}
\tau_N := \inf\setc{ t\geq 0}{\norm{X^{(N)}(t)}{1}\geq N} \label{eq:tauN}
\end{equation}
and set
\begin{equation}
\tau := \lim_{N \to \infty} \tau_N.\label{eq:tau}
\end{equation}
We start with the following preparatory Lemma.
\begin{lem}\label{lem:monL} The stopping times defined in \eqref{eq:tauN} and \eqref{eq:tau} have the following properties:
\begin{enumerate}
 \item For all $k\in\N$ the equality $X^{(N)}(t) = X^{(N+k)}(t)$ holds a.s. for $t \in \llbrak 0, \tau_N \rrbrak$.
  \item The stopping time $\tau$ is strictly positive.
  \end{enumerate}
\end{lem}
\begin{proof}
By definition of $h_N$ it holds that $B_{N+1}(X^{(N)}(s)) = B_N(X^{(N)}(s))$ and $C_{N+1}(X^{(N)}(s)) = C_N(X^{(N)}(s))$ for $s \in \llbrak 0, \tau_N \rrbrak $. 
Hence, using the localization property of the stochastic convolution (cf. \cite[Appendix A]{BrzeBeam} and \cite[Lemma 5.1]{weisMaxRegEvEq}), 
  \begin{gather*}
    \begin{split}
      X^{(N)}({t}) 
      &= S_{t} X_0 + \int_0^{t} S_{t-s}B_N(X^{(N)}(s)) \d s + \int_0^{t} S_{t-s} C_N(X^{(N)}(s))\d W(s) \\
      &= S_{t} X_0 + \int_0^{t} S_{t-s}B_{N+1}(X^{(N )}(s)) \d s + \int_0^{t} S_{t-s}C_{N+1}(X^{(N)}(s))\d W(s)
    \end{split}
     \end{gather*}
  on $\llbrak 0, \tau_N \rrbrak$. Now, uniqueness of the truncated solutions yields $X^{(N)}(t) = X^{(N+1)}(t)$ almost surely on $\llbrak 0, \tau_N \rrbrak$. For general $k\in \N$ the argument can be iterated.
  
  To show that $\tau$ is strictly positive, note that it follows from path-wise continuity of $X^{(N)}$ that
  \begin{multline*}
      \PP\left[ \tau_{N} >0 \right] 
      = \lim_{k\rightarrow \infty} \PP\left[ \tau_{N} >\frac1{k}\right]
      = \lim_{k\rightarrow \infty} \PP\left[  \sup_{0\leq t\leq \sfrac1{k}} \norm{X^{(N)}(t)}{1} < N\right]
      = \PP\left[ \norm{X_0}{1} < N\right].
  \end{multline*}
  From the first part of the proof, $(\tau_N)$ is increasing. Hence,
  \begin{multline*}
    \:\qquad\PP\left[ \tau > 0\right] 
    = \lim_{N\rightarrow \infty}  \PP\left[ \tau_{N} >0 \right] 
   = \lim_{N\rightarrow \infty}  \PP\left[ \norm{X_0}{1} < N\right]
   = \PP\left[ \norm{X_0}{1} <\infty\right]
   =1,\qquad
  \end{multline*}
showing positivity of $\tau$.
\end{proof}

\begin{proof}[Proof of Theorem~\ref{thm:EUCloc}]
For $x \in \R_+$ and $t \in \llbrak 0, \tau \llbrak$ we set
  \begin{align}
    X(t,x) :=& \lim_{N\rightarrow \infty} X^{(N)}(t,x), \label{eq:Xm}.
  \end{align}
The limit exists, since for almost every $\omega \in \Omega$ the sequence $(X^{(N)}(\omega; t,x))_{N \in \N}$ is eventually constant for each $t \in \llbrak 0,\tau \llbrak$ by Lemma~\ref{lem:monL}. It follows immediately that $t \mapsto X(t,x)$ is a.s. continuous on each $\llbrak 0,\tau_N \rrbrak $ and hence also continuous on $\llbrak 0,\tau \llbrak$. Moreover, we may now rewrite $\tau_N$ as
\[ \tau_N := \inf\setc{ t\geq 0}{\norm{X(t)}{1}\geq N}.\]
Continuity of $X$ then implies that the sequence $\tau_N$ is in fact \emph{strictly} increasing and hence that $\tau$ is predictable. Moreover, by definition of $\tau$ we have
\[\lim_{t \to \tau} \norm{X(t)}{1} = \lim_{N \to \infty} \norm{X(\tau_N)}{1} = \infty\]
on $\{\tau < \infty\}$.

We focus on the claim that $X$ solves \eqref{eq:seebd}. By Lemma~\ref{lem:monL}  it holds that  $X^{(N)}(t) =  X(t)$ on $t \in \llbrak 0, \tau_N \rrbrak$. Moreover, by construction of $B_N$ and $C_N$ we get $B(X(t)) = B_{N}(X(t))$ and $C(X(t)) = C_N(X(t))$ on $\llbrak 0, \tau_N \rrbrak$. Thus,
  \begin{gather}
    \begin{split}
      &\quad S_t X_0 + \int_0^{t} S_{{t}-s}B(X(s)) \d s + \int_0^{t} S_{{t}-s}C(X)(s))\d W(s) \\
      &= S_{t} X_0 + \int_0^{t} S_{{t}-s}B_{N}(X^{(N)}(s)) \d s + \int_0^{t} S_{{t}-s}C_N(X^{(N)}(s))\d W(s) \\
      &= X^{(N)}({t})\\
      &= X({t})
    \end{split}
   \end{gather}
holds, and $X$ is a mild solution of \eqref{eq:seebd} on $\llbrak 0, \tau_N \rrbrak$. Since $N$ was arbitrary $X$ is a mild solution on $\llbrak 0, \tau \llbrak$ as claimed.
  To show uniqueness, let $\widetilde X$ be another local mild $E_1$-solution of \eqref{eq:seebd} on some stochastic interval $\llbrak 0, \widetilde \tau \llbrak$. For $N\in \N$ we introduce the stopping time
  \[ \widetilde\tau_N := \inf\left\{ t\geq 0\cond \norm{\widetilde{X}(t)}{1} \geq N \text{ or } \norm{\vphantom{\int}X(t)}{1} \geq N\right\}\wedge \widetilde\tau.\]
Clearly, it holds that
  \[ \lim_{N\rightarrow\infty} \widetilde\tau_N = \widetilde\tau\wedge\tau, \qquad a.s.\] 
In addition, for $t \in \llbrak 0, \widetilde \tau_N \rrbrak$ it holds that 
  \[B(\widetilde X(t)) = B_N(\widetilde X(t)),\quad\text{and}\quad C(\widetilde X(t)) = C_N(\widetilde X(t)).\]
 As above, we derive that $\widetilde X$ is a mild solution of the truncated equation on $ \llbrak 0,\widetilde\tau_N \rrbrak$. The path-wise uniqueness claim of Theorem~\ref{thm:EUC} implies $\widetilde X(t) = X(t)$ for all $t \in \llbrak 0,\widetilde\tau_N \rrbrak$ and by arbitrariness of $N$ also for $t \in \llbrak 0, \widetilde \tau \wedge \tau \llbrak$. Assume now that $\tau < \widetilde \tau$ on a set of positive probability. Then $\lim_{t\nearrow\tau} \norm{X(t)}{1} = \infty$ on $\{\tau<\infty\}$ leads to a contradiction to the continuity of $\widetilde X$. Hence $X$ is unique and $\llbrak 0, \tau \llbrak$ is maximal. 
\end{proof}

\subsection{More global solutions and strong solutions}
The results from above can be easily extended in two directions: First, we show the existence of global solutions under more general conditions, second we show that all mild solutions obtained in this section are in fact strong solutions.

\begin{cor}[More global solutions]\label{cor:LGglobal} Let the assumptions of Theorem~\ref{thm:EUCloc} hold true, but with the local growth condition \eqref{eq:BDlg} replaced by the global growth condition \eqref{eq:LG}. Then the solution $X$ is global, i.e. $\tau = \infty$ a.s.
\end{cor}
\begin{proof}
By Theorem~\ref{thm:EUCloc} we know that a local solution $X$ to \eqref{eq:seebd} exists on a maximal stochastic interval $\llbrak 0, \tau \llbrak$ and that $\lim_{t \to \tau} \norm{X(t)}{1} = \infty$ on $\{\tau < \infty\}$. Moreover, we know that the stopping time $\tau$ is the limit of a sequence of stopping times $\tau_N < \tau$, that $X$ coincides with the (global) solution $X^N$ of the truncated equation \eqref{eq:SEEqN} on the stochastic interval $\llbrak 0,  \tau_N \rrbrak$ and that $\norm{X^N(\tau_N)}{1} \ge N$ for all $N \in \N$. Finally, observe that the coefficients $B_N, C_N$ of the truncated equation satisfy the same growth bound as the coefficients of the original equation, i.e. 
\begin{equation}\label{eq:growth_bound_repeat}
\norm{B_N(Y)}{\alpha}  + \norm{C_N(Y)}{\HS(U,E_1)} \le M \left(1 + \norm{Y}{1}\right), \qquad \text{for all $N \in \N$}
\end{equation}
with $M$ independent of $N$.\\
If $\{\tau < \infty\}$ has measure zero, then $\tau  = \infty$ a.s. and the proof is finished. Therefore assume, aiming for a contradiction,  that $\P{\tau < \infty} = 2\epsilon > 0$. By monotone convergence it follows that there exists $T > 0$ such that $\P{\tau < T} \ge \epsilon$ from which it follows that also $\P{\tau_N \le T} \ge \epsilon$
for all $N \in \N$. Hence
\begin{equation}\label{eq:upper_bound_XN}
\E{\sup_{0 \le t \le T} \norm{X^N(t)}{1}^{2p}} \ge N^{2p} \P{\tau_N \le T} \ge N^{2p} \epsilon
\end{equation}
holds. On the other hand, applying the growth bound \eqref{eq:growth_bound_repeat} and Lemma~\ref{lem:EsupX} to each $X_N$ it follows that 
\[ \E{\sup_{0 \le t \le T} \norm{X^N(t)}{1}^{2p}}  \le K_{p,T} \left(1 + \E{\norm{X_0}{0}^{2p}}\right),\]
with the right hand side independent of $N$. Combining with \eqref{eq:upper_bound_XN} and choosing $N$ large enough, the desired contradiction is obtained.
\end{proof}

\begin{cor}[Strong solution]
  \label{thm:eus}
  Under the assumptions of Theorem~\ref{thm:EUC} and Theorem~\ref{thm:EUCloc} respectively, the processes $X^{(N)}$, $N\in\N$ and $X$ are also the unique strong solution of respectively~\eqref{eq:SEEqN} and~\eqref{eq:seebd}. 
\end{cor}
\begin{proof}
Let $(X,\tau)$ be the unique mild solution from Theorem~\ref{thm:EUCloc}, and $X^{(N)}$, $\tau_N$ and respectively $B_N$ and $C_N$ be the solution, stopping times and parameters of the truncated equation~\eqref{eq:SEEqN}, corresponding to~\eqref{eq:seebd}. By the natural embedding we can identify the $E_1$-paths of $X^{(N)}$ with paths in $E$. Further, Kuratowski's theorem (cf. \cite{kuratowski}) implies for the corresponding Borel $\sigma$-algebras that $\borel{E_1} = \borel{E}\cap E_1$ and hence that we can extend $B_N$ trivially to a Borel function on $E$ without affecting the regularity properties that $B_N$ has on $E_1$. Writing down the Hilbert-Schmidt norm one immediately observes that also
\[ \HS(U, E_1) \hookrightarrow \HS(U, E).\]
Both are seperable Hilbert spaces so that we can argue in the same way to extend $C_N$ to a Borel function from $E$ into $\HS(U, E)$. The stochastic evolution equations now fit in the framework of \cite[Appendix F]{roecknerprevot}, where sufficient conditions for obtaining weak and strong solutions from mild ones are given. Proving the corollary now simply amounts to showing that these conditions are satisfied.

First, recall~\eqref{eq:EintBCfinite} which yields for all $T\geq 0$ that
\[\int_0^T \norm{B_N(X^{(N)}(s)}{} + \norm{C_N(X^{(N)}(s)}{\HS(U;E)} \d s < \infty,\]
$\PP$-almost surely and from equivalence of the norms of $E_1$ and $\dom(A)$ we get a.\,s.
\[ \int_0^T \norm{X^{(N)}(s)}{} + \norm{AX^{(N)}(s)}{} \d s< \infty.\]
For the step from mild to weak solutions we also need to verify for all  $g\in \dom(A^*)$
\begin{equation}
  \int_0^T\EE\int_0^t  \norm{\scal{S_tC_N(X^{(N)}(s))}{A^*g}{E}}{\HS(U,\R)}^2 \d s \d t < \infty.    \label{eq:m2wC}
\end{equation}
Recall that $A$ generates a strongly continuous semigroup on $E$. Using Cauchy-Schwarz inequality we then get for any complete orthonormal system $(e_k)$ of $U$ and $0\leq s\leq t$
  \begin{gather*}
    \begin{split}
      \norm{\scal{S_{t-s}\cC_N(X^{(N)}(s))}{A^*g}{E}}{\HS(U,\R)}^2 
      &= \sum_{k=1}^\infty \abs{\scal{S_{t-s}C_N(X^{(N)}(s))e_k}{A^*g}{E}}^2\\
      &\leq  \norm{S_{t-s}C_N(X^{(N)}(s))}{\HS(U; E)}^2 \norm{A^*g}{E}^2\\
      &\leq  \norm{C_N(X^{(N)}(s))}{\HS(U,E)}^2 \norm{A^*g}{E}^2.
    \end{split}
  \end{gather*}
  From~\eqref{eq:EintBCfinite} (or boundedness of $C_N$) it follows that~\eqref{eq:m2wC} indeed holds true. 
  As we have seen in the proof of the continuity part of Theorem~\ref{thm:EUC}, 
  \[ t\mapsto \int_0^t S_{t-s} B_N(X^{(N)}(s))\d s,\quad\text{and}\quad t\mapsto \int_0^tS_{t-s} C_N(X^{(N)}(s)) \d W_s\]
  are continuous, and therefore predictable.
  Hence,  all integrability and measurability assumptions which are needed to apply~\cite[Prop. F.0.4 and F.0.5]{roecknerprevot} are satisfied and we conclude that the mild solution $X^{(N)}$ is also a weak and strong solution of the truncated equation \eqref{eq:SEEqN}. Applying the the same localization argument as above, but now on $E$, the result translates to $X$. 
  
For the uniqueness claim, suppose first that $Y^{(N)}$ is another global strong solution of the truncated equation~\eqref{eq:SEEqN}. Applying the results in~\cite[Appendix F]{roecknerprevot} we get that $Y^{(N)}$ is also an $E$-mild solution. Since $B_N$ and $C_N$ are bounded, we get that
  \[\int_0^t \norm{S_{t-s}B_N(Y^{(N)})(s)}{E_1} + \norm{S_{t-s}C_N(Y^{(N)}(s))}{\HS(U;E_1)}^2\d s <\infty\]
  almost surely, so that $Y^{(N)}$ is even an $E_1$-mild solution. Hence, the uniqueness part of Theorem~\ref{thm:EUC} yields $Y^{(N)} = X^{(N)}$ almost surely. 
  Finally, for another local strong solution $(Y, \varsigma)$, truncating with respect to the $E_1$-norm yields $Y = X$ on $\llbrak 0, \varsigma \wedge \tau \llbrak$, almost surely. Since $\llbrak 0, \tau\llbrak$ is maximal for the $E_1$-mild solution, we obtain that $\varsigma \leq \tau$ almost surely.
\end{proof} 


%% file: proofEUStopping.tex
The goal of this section is to reformulate the free boundary problem \eqref{eq:fbp} as the abstract evolution equation \eqref{eq:seebd}. We start start by identifying the appropriate function spaces and introduce
\begin{align*}
  \L^2 &:= L^2(\R_+) \oplus L^2(\R_+) \oplus \R,\qquad \H^k:= H^k(\R_+)\oplus H^k(\R_+) \oplus \R,
  \end{align*}
 where, as usual, $L^2$ denotes the Lebesgue space, $H^k$ the $k$-th order Sobolev space. 
 Recall that $\oplus$ denotes the direct sum of Hilbert spaces, i.e. the scalar product on $\L^2$ is defined through the scalar product on $L^2(\R_+)$ by
 \[\scal{(u_1,u_2,x)}{(v_1,v_2,y)}{\L^2} = \scal{u_1}{v_1}{L^2(\R_+)} + \scal{u_2}{v_2}{L^2(\R_+)} + xy,\]
 and similarly for the space $\H^k$.
 
 Recall the definitions of the operators $\cA$, $\cB$, $\cC$ that were given in equations~\eqref{eq:coefA}, \eqref{eq:coefB}, \eqref{eq:coefC} in terms of $\mu_\pm$, $\sigma_\pm$, $\rho$ and $\zeta$. To define the domain $\DA$ of $\cA$ we set 
\[ D := H^2(\R_+) \cap H^1_0(\R_+), \quad \text{and} \quad \DA = D \times D \times \R \subseteq \H^2.\]
Finally, the space $\dom(\cA)$ shall be equipped with the graph norm
 \[ \norm{u}{\cA} := \norm{u}{\L^2} +\norm{\cA u}{\L^2},\;\quad u\in \dom(\cA).\]
 Note that on $\DA$, the graph norm is equivalent to the $\H^2$-norm, as can be seen from integration by parts and the Cauchy-Schwarz inequality. Moreover, $\DA$ is a closed subset of $\H^2$.\\
 
 To provide the connection with the results of section~\ref{sec:seeqs}, we set
 \[E_0 = \L^2, \qquad E_1 = \DA \subseteq \H^2.\]
 As the norms of $\DA$ and $\H^2$ are equivalent we may use either one to topologize $E_1$.
 The following result holds true for interpolation between $\DA$ and $\L^2$:
 \begin{lem}
   For $\alpha \in (0,1/4)$ it holds that,
   \[E_\alpha = \H^{2\alpha},\]
   with equivalent norms.
 \end{lem}
 This follows from the fact that $\dom((c-\Delta)^\alpha) = H^{2\alpha}(\R_+)$ for the Dirichlet Laplacian $\Delta$, iff $\alpha<\sfrac14$. To our knowledge, this was shown first in~\cite{grisvardComm}, but see also~\cite[Ch.1 , Thm. 11.6]{lionsmagenes1}. By the structure of $\DA$ and  $\L^2$, this directly lifts to $E_\alpha$. This result can be understood in the sense that the boundary conditions, which distinguish $\DA$ from $\H^2$ `are lost' during interpolation exactly at $\alpha = 1/4$. 

\begin{lem}
  The operator $\cA$, defined in \eqref{eq:coefA}, satisfies Assumption~\ref{a:A}.
\end{lem}
\begin{proof}The Dirichlet Laplacian on $L^2(\R_+)$ is a self-adjoint operator. This property is inherited by $\cA$, which is hence a self-adjoint operator 
  . Moreover, 
\[-\scal{\cA u}{u}{\L^2} = \norm{\ddx u_1}{L^2(\R_
+)}^2 + \norm{\ddx u_2}{L^2(\R_
+)}^2 + c \norm{u}{\L^2}^2 \ge c \norm{u}{\L^2}^2\]
for all $u \in \DA$ and hence $-\cA$ is \emph{positive} in the sense of \cite{lunardiInterpol}. By \cite[Lem.~4.31]{lunardiInterpol}  it follows that $(-c, \infty)$ is contained in the resolvent set of $\cA$ and the resolvent satisfies
\[\norm{R(\lambda,A)}{\L^2} \le \frac{1}{c + \lambda}, \qquad \lambda > -c.\]
Choosing $M > \max(1,1/c)$, the estimate \eqref{eq:resolvent_eq} follows and Assumption~\ref{a:A} is satisfied.
\end{proof}

\begin{lem}\label{lem:assB}
Suppose that $\mu_{\pm}$, $\sigma_{\pm}$, $\rho$ and $\zeta$ satisfy Assumptions~\ref{a:mu},\ref{a:sigma}, \ref{a:rho} and \ref{a:zeta}. Then the operators $\cB(u)$ and $\cC(u)$, defined in \eqref{eq:coefB} and \eqref{eq:coefC} satisfy the local Lipschitz Assumption~\ref{a:SEElocal}.
\end{lem}
\begin{proof}
We decompose $\cB(u)$ into
\begin{equation}
  \label{eq:Bdecomp}
  \cB(u) := \cB_\mu(u) + \cB_\rho(u) + c u, \quad u\in \DA,
\end{equation}
where \begin{equation*}
  \cB_\mu(u) = \begin{pmatrix}
    \mu_+(.,u_1,  \ddx u_1) \\
    \mu_-(- .,u_2,  \ddx u_2) \\
    0
  \end{pmatrix}, \quad \cB_\rho(u) = \rho(\cI(u)) \bar\nabla u \quad \text{and} \quad \bar\nabla u := \begin{pmatrix}
    \ddx u_1  \\
    - \ddx u_2 \\
    1
  \end{pmatrix}
\end{equation*} 
Recall from \eqref{eq:I} that $\cI(u)$ is the vector of boundary values given by $\cI(u) = \left(\ddx u_1(t,0+), \ddx u_2(t,0+)\right)$. 
%
The trace operator is known to be continuous on $H^{2}(\R_+)$, see~\cite{lionsmagenes1}, and hence, by equivalence of norms we have that $\cI\in L(\H^2, \R^2)$ with operator norm $K_{\cI}$, say. Denote by $B_{N+1}$ the (closed) ball of radius $N+1$ in $\H^2$. The image of $B_{N+1}$ under $\cI$ is closed and bounded, hence a compact subset of $\R^2$.      
        By Assumption~\ref{a:rho}, the function $\rho: \R^2 \mapsto \R$ is locally Lipschitz continuous and hence Lipschitz on any compact subset of $\R^2$. Thus, we find a constant $L_{\rho,N}$ such that
        \[|\rho(\cI(u)) - \rho(\cI(w))| \le L_{\rho,N} |\cI(u) - \cI(w)| = L_{\rho,N} K_{\cI} \norm{u-w}{\H^2}, \quad \forall\,u,w \in B_{N +
        1}.\]
        By a similar argument and using only continuity of $\rho$ instead of the Lipschitz property, we find $M_{\rho,N}$ such that
        \[\sup_{u \in B_{N+1}} |\rho(\cI(u))| \le M_{\rho,N} K_{\cI}.\]
                
        Finally, for any $u, w \in B_{N + 1}$ and setting $L_N = K_\cI (M_{\rho,N} + (N+1)L_{\rho,N})$ we obtain        
        \begin{align*}
        \norm{\rho(\cI(u)) \bar\nabla u - \rho(\cI(w)) \bar\nabla w}{\H^1} &\le |\rho(\cI(u))| \norm{\bar\nabla u - \bar\nabla w}{\H^1} + \norm{\bar\nabla u}{\H^1} |\rho(\cI(u)) - \rho(\cI(w))| \le \\
        & \le M_{\rho,N} K_{\cI} \norm{u-w}{\H^2} + (N+1) L_{\rho,N} K_{\cI} \norm{u-w}{\H^2} = \\
        &= L_N \norm{u-w}{\H^2}.
        \end{align*}
	Now let $\alpha < 1/4$. By the continuous embedding $\H^1 \hookrightarrow \H^{2\alpha}$ and the equivalence of norms on $\H^2$ and  $\DA$ we may find $K_\alpha$ such that also
	\[\norm{\rho(\cI(u)) \bar\nabla u - \rho(\cI(w)) \bar\nabla w}{\H^{2\alpha}} \le K_\alpha L_N \norm{u-w}{\DA}.\]
	Recalling that $\H^{2\alpha} = E_\alpha$ and $\DA = E_1$, this yields that $\cB_\rho$ is Lipschitz as a mapping from bounded subsets of $E_1$ to $E_\alpha$.

	It remains to show the same properties for $\cB_\mu$ and for $\cC$, which we delay to Appendix~\ref{A:nem} and \ref{A:noise}, respectively. \qedhere
	
 \end{proof}
 In particular, with the results from section~\ref{sec:seeqs} we get a unique solution $X = (u_1,u_2, x_*)$ of~\eqref{eq:SEEq} on the maximal interval $\llbrak 0,\tau\llbrak$. Applying the following lemma to Corollary~\ref{cor:LGglobal} yields global existence, under global growth assumptions.

\begin{lem}\label{cor:BLG}
 Suppose that in addition to the assumptions of Lemma~\ref{lem:assB} also Assumption~\ref{a:mslg} holds and $\rho$ is bounded. Then the operators $\cB$ and $\cC$ satisfy the global growth bound \eqref{eq:LG}. 
\end{lem}

\begin{proof}Decompose $\cB$ into $\cB_\mu$ and $\cB_\rho$ as in the proof of Lemma~\ref{lem:assB}. Using Assumption~\ref{a:mslg} we get the point-wise estimate
    \[\abs{B_\mu(u)_i} \leq \left( a + b(u_i, \ddx u_i)) (\abs{u_i}+\abs{\ddx u_i})\right),\qquad i\in \{1,2\}\]
    where $a\in L^2(\R_+)$ and $b \in L^\infty(\R)$. Taking $\L^2$-norms yields 
   \begin{equation}
          \norm{B_\mu(u)}{\L^2} \leq 2 \left( \norm{a}{L^2}+ \norm{b}{\infty} \norm{u}{\H^1}\right). \label{eq:bdPhi}
    \end{equation}
    For the first weak derivative we extract from the proof of Theorem~\ref{thm:Ncont} (cf. Eq.~\eqref{eq:DN}.) 
    \begin{equation}
      \begin{split}
        \tfrac{\d}{\d x} \mu_+(x, u_1(x), \ddx u_1(x)) = & \ddx \mu_+(x, u_1(x), \ddx u_1(x)) \\
        &+ \ddy \mu_+(x,u_1(x), \ddx u_1(x)) \ddx u_1(x) \\
        &+ \ddz \mu_+(x,u_1(x), \ddx u_1(x)) \ddxx u_1(x),\;x\in \R_+.
      \end{split}
    \end{equation}
    
    Whereas the first summand admits a bound similar to~\eqref{eq:bdPhi}, part (b) of Assumption~\ref{a:mslg} yields for some $\tilde b\in L^\infty(\R^2;\R)$,
    \[ \abs{\ddy \mu_+(x,u_1(x), \ddx u_1(x)) \ddx u_1(x)} \leq \tilde b(u_1(x),\ddx u_2(x)) \abs{\ddx u_1}\]
    and 
    \[ \abs{\ddz \mu_+(x,u_1(x), \ddx u_1(x)) \ddxx u_1(x)} \leq \tilde b(u_1(x),\ddx u_2(x)) \abs{\ddxx u_1}.\]
    Of course, the same holds for $\mu_-$ and $u_2$ so that we can summarize
    \begin{equation}
      \norm{\tfrac{\d}{\d x} \cB_\mu(u)}{\L^2} \leq 2 \left( \norm{a}{L^2}+ \norm{b}{\infty} \norm{u}{\H^1}\right) + 2 \norm{\tilde b}{\infty} \norm{u}{\H^2}.\label{eq:bdDPhi}
    \end{equation}
    Collecting all the estimates and using $\H^2 \hookrightarrow \H^1$ we get a constant $M_\mu$, depending on $a$, $b$ and $\tilde b$ only, such that 
    \begin{equation}
      \norm{\cB_\mu(u)}{\H^1} \leq  M_\mu \left(1 + \norm{u}{\H^2}\right).\label{eq:BLGlocal}
    \end{equation}
Using the assumption that $\rho$ is bounded by a constant $M_\rho$ we easily estimate
	\[\norm{\cB_\rho(u)}{\H^1} = \norm{\rho(\cI(u)) \bar\nabla u}{\H^1} \le M_\rho \norm{u}{\H^2}.\]
Combining with \eqref{eq:BLGlocal} we obtain 
	\[\norm{\cB(u)}{\H^1} \le (M_\mu + M_\rho + c) \left(1 + \norm{u}{\H^2}\right).\]
	Using the continuous embedding $\H^1 \hookrightarrow \H^{2\alpha}$ and the equivalence of norms on $\cH^2$ and  $\DA$ as in the proof of Lemma~\ref{lem:assB} yields the global growth bound~\eqref{eq:LG} for $\cB$.

        To show the analogous growth bound for $\cC$, observe that the following equalities hold for all $x \in \R_+$ and $u\in \DA \subset \H^2$:
        \begin{align*}
          \ddx \sigma_+(x, u_1(x)) &= \ddx \sigma_+^1(x) + \ddx \sigma_+^2(x)u_1(x) + \sigma_+^2(x) \ddx u_1(x)\\
          \ddxx \sigma_+(x, u_1(x)) &= \ddxx \sigma_+^1(x) + \ddxx \sigma_+^2(x)u_1(x)\\
          &\qquad\qquad\qquad+2\ddx\sigma_+^2(x)\ddx u_1(x) + \sigma_+^2(x) \ddxx u_1(x).
        \end{align*}
        Hence, there exists a constant $K>0$ such that
        \begin{equation*}
          \norm{\sigma_+(.,u_1(.))}{H^2(\R_+)} \leq \norm{\sigma_+^1}{H^2(\R_+)} + K \norm{\sigma_+^2}{C^2(\R_+)} \norm{u_1}{H^2(\R_+)}.
        \end{equation*}
        We apply the same argument to $\sigma_-(-.,u_2(.))$ to obtain
        \begin{equation}
          \norm{\begin{pmatrix} \sigma_+(.,u_1(.))\\ \sigma_-(-.,u_2(.)) \\ 0 \end{pmatrix}}{\cA} \leq K_\sigma \left( 1+ \norm{u}{\cA}\right),\label{eq:nsigmaLG}
        \end{equation}
        for a constant $K_\sigma$, depending on $\sigma_+$ and $\sigma_-$ only. By Assumption~\ref{a:sigma} $\sigma_{\pm}(.,u_1(.))$ satisfies Dirichlet boundary conditions at $0$ and we may apply Lemma~\ref{lem:HSest} to obtain
        \begin{equation*}
          \norm{\cC(u)}{\HS(U;\DA)} \leq K\left(\sup_{x\in\R} \sum_{i=1}^2\norm{\zeta^{(i)}(x,.)}{L^2}\right) \norm{\begin{pmatrix} \sigma_+(.,u_1(.))\\ \sigma_-(-.,u_2(.)) \\ 0 \end{pmatrix}}{\cA},
        \end{equation*}
        which together with~\eqref{eq:nsigmaLG} yields the desired linear growth bound. 
\end{proof}

Finally, we prove a refined result on the blow-up behavior of the solution $X(t)$ at the stopping time $\tau$. We show that under Assumption~\ref{a:mslg} a finite-time blow-up of $X(t)$ can only happen if the boundary values $\cI(X(t))$ themselves blow up. For $N\in \N$ we introduce the stopping times 
\begin{align*}
  \tau_{\circ,N} &:= \inf\setc{t>0}{ t < \tau, \abs{\cI(X(t))} \geq N}.
\end{align*}
and set
\[\tau_\circ = \limsup_{N \to \infty} \tau_{\circ,N}.\]
Recall at this point the convention that $\inf \emptyset = \infty$. From Theorem~\ref{thm:EUCloc} we know that
\begin{equation}
  \label{eq:explA}
  \lim_{t\nearrow \tau} \norm{X(t)}{\cA} = \infty,\quad\text{a.\,s. on}\;\{\tau <\infty\}.
\end{equation}
Since a blow up of $\abs{\cI(X(t))}$ implies a blow-up of the norm $\norm{X(t)}{\cH^2}$, and hence also of $\norm{X(t)}{\cA}$ we obtain that $\tau \le \tau_\circ$ a.s. and hence only two events are possible: Either
\begin{itemize}
\item $\tau = \tau_\circ$, i.e. a blow-up of $X(t)$ coincides with the blow up of the boundary values $\cI(X(t))$, or
\item $\tau < \infty$, but $\tau_\circ = +\infty$, i.e. a blow-up of $X(t)$ occurs \emph{without} simultaneous blow-up of its boundary values. 
\end{itemize} 
The following theorem shows that  Assumption~\ref{a:mslg} rules out the second case:
 \begin{thm}\label{th:tau0eqtau}
   If in addition to the assumptions of Lemma~\ref{lem:assB} also Assumption~\ref{a:mslg} holds, then
   \[ \P{\tau_\circ = \tau} = 1.\]
 \end{thm}
 \begin{proof}
   Because of $\tau_\circ\geq \tau$ and maximality of $\tau$,~\eqref{eq:explA}, it suffices to show
  \begin{equation}
    \label{eq:32}
    \lim_{t\nearrow \tau_{\circ,N} \wedge \tau} \norm{X(t)}{\cA} <\infty,\qquad\text{on }\{\tau <\infty\}.
  \end{equation}
  Indeed, this yields $\tau_{\circ,N}<\tau$ on $\{\tau<\infty\}$ and thus $\tau_{\circ}\leq \tau$ almost surely.
  For $N\in \N$, let $h_N:\R\to\R$ be the truncation function defined in~\eqref{eq:truncfct} and define
  \[\rho_N(x,y) := \rho(x,y) h_N(\abs{(x,y)}),\quad  (x,y)\in \R^2,\,N\in \N.\]
  Then, due to Theorem~\ref{thm:globalsol}, equation~\eqref{eq:SEEq} with $\rho$ replaced by $\rho_N$ admits a unique global solution denoted by $X_N$. Since $\rho = \rho_N$ on the ball of radius $N$, we have $X(t) = X_N(t)$ on $\llbrak 0,\tau_{\circ,N}\wedge \tau \llbrak$. Since $X_N$ is $\DA$-continuous, we have
  \begin{equation*}
    \lim_{t\nearrow \tau_{\circ,N} \wedge \tau} \norm{X(t)}{\cA} = \lim_{t\nearrow \tau_{\circ,N}\wedge \tau} \norm{X_N(t)}{\cA} <\infty,\quad \text{on }\{\tau<\infty\}
  \end{equation*}
  for all $N \in \NN$ and the proof is complete.
 \end{proof}

%% file: itoformula2a_classic.tex
As the last step towards a complete proof of the main result Theorem~\ref{thm:sfbp} we make the transformation \eqref{eq:trafo} to the fixed-boundary equation \eqref{eq:fbpt} rigorous. Since the equation is stochastic and its solution not differentiable in time, the classic chain rule cannot be applied. As an alternative we could use Ito's formula, which requires the transformation map to be $C^2$. It turns out that the transformation is only $C^1$, but linearity in its first argument and the bounded variation of $x_*(t)$ are in combination sufficient to make a stochastic version of the chain rule work.
\subsection{Stochastic chain rule}
For a given cylindrical Wiener process $W$ on $U$ we consider the $H^1(\RR)\oplus\RR$-continuous process $(v,x)$ on $\llbrak 0,\tau \llbrak$, for a predictable stopping time $\tau$, such that,
\begin{equation}\label{eq:SDEv}
  \left\{\begin{split}
      v_t &= v_0 + \int_0^t \mu_s \d s + \int_0^t \sigma_s \d W_s,\\
      x(t) &= x_0 + \int_0^t \dot x_s \d s,
    \end{split} \right.
\end{equation}
where $\mu: \llbrak 0,\tau \llbrak \rightarrow L^2(\R)$ and $\sigma: \llbrak 0, \tau\llbrak \to \HS(U,L^2(\R))$ are predictable processes, and $ \dot x: \llbrak 0,\tau\llbrak \rightarrow \RR$ is continuous and adapted.

For $x \in \R$ we define the shift operator $\theta_x$ acting on $L^2(\R)$ as
\begin{equation}
\theta_x: L^2(\R) \to L^2(\R); \quad f(.) \mapsto f(x + .) \label{eq:shift}.
\end{equation}
Observe that the shift operator is a linear isometry and hence continuous. It is obvious that the shift operators $(\theta_x)_{x \in \R}$ form a group under composition, i.e. $\theta_x \theta_\xi = \theta_{x + \xi}$; in fact this group is strongly continuous in $L^2(\R)$, in the sense that
\begin{equation}\label{eq:theta_group}
  \lim_{x \to 0} \norm{\theta_x f - f}{L^2} = 0, \qquad \text{for all $f \in L^2(\R)$,}
\end{equation}
see e.\,g.~\cite[Section VII.4]{wernerFunkana}. The same properties hold true for the restriction of $(\theta_x)_{x \in \R}$ to the Sobolev space $H^1(\R)$. Finally, consider the function
\begin{equation}
  \trafo:  L^2(\R_+) \oplus \R \to L^2(\R), \quad (v,x)  \mapsto \theta_x v \label{F}
\end{equation}
which formally transforms the solution $(u,x_*)$ of the fixed boundary problem \eqref{eq:fbpt} into the solution $(F(v,x_*),x_*)$ of the moving boundary problem \eqref{eq:fbp}. 


We start with a Lemma on some uniform continuity estimates for the shift operator on $L^2(\R)$. 
\begin{lem}\label{lem:uniform_estimate}
Let $T>0$, $s \mapsto f_s$ and $s \mapsto g_s$ be continuous functions from $[0,T]$ to $L^2(\R)$ and $\HS(U,L^2(\R))$ respectively, and let $\theta_x$ be the shift operator on $L^2(\R)$. Then the following holds:
\begin{enumerate}[label=(\alph*),itemsep=0.5em]
\item For every $\epsilon > 0$ there exists $\delta > 0$ such that
\[\norm{\theta_h f_s - f_t}{L^2(\R)} \le \epsilon\]
for all $|h| \le \delta$ and $s,t \in [0,T]$ with $|s-t| \le \delta$.
\item For every $\epsilon > 0$ there exists $\delta > 0$ such that
\[\norm{\theta_h g_s - g_t}{\HS(U,L^2(\R))} \le \epsilon\]
for all $|h| \le \delta$ and $s,t \in [0,T]$ with $|s-t| \le \delta$.
\end{enumerate}
\end{lem}
\begin{rmk}
  \label{rmk:shiftHS}
  Applying the lemma to the constant function $g_t\equiv g$, $\theta$ can be considered as a strongly continuous group on $\HS(U;L^2(\R))$.
\end{rmk}
\begin{proof}
For claim (a), note that due to the continuity of $s \mapsto f_s$ there exists $N := N(\epsilon) \in \N$ and $t_1, \dotsc, t_N \in [0,T]$ such that
\[\min_{i \in \{1, \dotsc, N\}} \norm{f_s - f_{t_i}}{L^2(\R)} \le \frac{\epsilon}{3} \qquad \forall\,s \in [0,T],\]
i.e., we can find $N$ balls in $L^2(\R)$ around the points $f_{t_i}$ of radius $\frac{\epsilon}{3}$, which cover the whole range of $f_s$.
Moreover, due to the strong continuity of the group $(\theta_x)_{x \in \R}$ there exists $\delta > 0$ such that 
\[\norm{\theta_h f_{t_i} - f_{t_i}}{L^2(\R)} \le \frac{\epsilon}{3} \qquad \forall\,|h| \le \delta, i \in \{1, \dotsc, N\}.\]
The estimate
\begin{align*}
  \norm{\theta_h f_s - f_t}{L^2(\R)} &\le \min_{i \in \{1, \dotsc, N\}} \left\{ \norm{\theta_h f_{t_i} - f_{t_i}}{L^2(\R)} + \norm{f_s - f_{t_i}}{L^2(\R)} + \norm{f_t - f_{t_i}}{L^2(\R)}\right\} \le \\ 
  &\le \frac{\epsilon}{3} + \frac{\epsilon}{3}  + \frac{\epsilon}{3}  = \epsilon
\end{align*}
concludes the proof of the first claim.\\
We show (b) in a similar way; to alleviate notation we denote by $\norm{.}{HS}$ the Hilbert-Schmidt norm $\norm{.}{\HS(U,L^2(\R))}$. Due to the continuity of $s \mapsto g_s$ there exists $N := N(\epsilon) \in \N$ and $t_1, \dotsc, t_N \in [0,T]$ such that
\[\min_{i \in \{1, \dotsc, N\}} \norm{g_s - g_{t_i}}{HS}^2 \le \frac{\epsilon^2}{12}\qquad \forall\,s \in [0,T].\]
Denote by $(e_k)_{k \in \N}$ an arbitrary orthonormal basis of the (separable) Hilbert space $U$. Since $g_s$ is a Hilbert-Schmidt-operator for every $s \in [0,T]$, the series $\norm{g_{t_i}}{HS}^2 = \sum_k \norm{g_{t_i} e_k}{L^2(\R)}^2$ converges for every $i \in \{1, \dotsc, N\}$. Hence, there is $M := M(\epsilon) \in \N$ such that
\[\sum_{k = M+1}^\infty \norm{g_{t_i} e_k}{L^2(\R)}^2 \le \frac{\epsilon^2}{48} \qquad \forall \, i \in \{1, \dotsc, N\}.\]
Moreover, due to the strong continuity of the group $(\theta_x)_{x \in \R}$ there exists $\delta > 0$ such that 
\[\sum_{k=1}^M \norm{\theta_h (g_{t_i} e_k) - g_{t_i}e_k}{L^2(\R)}^2 \le \frac{\epsilon^2}{12}\qquad \forall\, |h| \le \delta,  i \in \{1, \dotsc, N\}.\]
Combining with the previous equation, we obtain
\begin{align*}
\norm{\theta_h g_{t_i} - g_{t_i}}{HS}^2 &= \sum_{k=1}^\infty \norm{\theta_h (g_{t_i} e_k) - g_{t_i}e_k}{L^2(\R)}^2 \le \\
&\le \sum_{k=1}^M \norm{\theta_h (g_{t_i} e_k) - g_{t_i}e_k}{L^2(\R)}^2 + 4 \sum_{k=M+1}^\infty \norm{g_{t_i}e_k}{L^2(\R)}^2 \le \\
& \le \frac{\epsilon^2}{12}  + 4 \frac{\epsilon^2}{48} = \frac{\epsilon^2}{6} \qquad \forall  \, i \in \{1, \dotsc, N\}.
\end{align*}
The estimate
\begin{align*}
\norm{\theta_h g_s - g_t}{HS}^2 &\le 3 \min_{i \in \{1, \dotsc, N\}} \left\{ \norm{\theta_h g_{t_i} - g_{t_i}}{HS}^2 + \norm{g_s -g_{t_i}}{HS}^2 + \norm{g_t - g_{t_i}}{HS}^2\right\} \le \\
&\le 3\left(\frac{\epsilon^2}{6} + \frac{\epsilon^2}{12}  + \frac{\epsilon^2}{12}\right)  = \epsilon^2
\end{align*}
concludes the proof.
\end{proof}

 For notational simplicity, we denote by $v'$ the weak derivative of $v\in H^1(\R)$.

\begin{lem}\label{prop:trafoProp}The transformation $\trafo$ from \eqref{F}, restricted to $H^1(\R)\oplus \R$, has the following properties:
  \begin{enumerate}
  \item \label{trafoPropC}  $F$ is a continuous mapping from $H^1(\R)\oplus \R$ to $H^1(\R)$;
  \item \label{trafoPropC1} $F$ is a continuously differentiable mapping from $H^1(\R)\oplus \R$ to $L^2(\R)$ with Fr\'echet derivative given by
    \begin{equation}\label{eq:frechet}
      D_{(v,x)} F(h,\xi) = \theta_x h+\xi \theta_x v' .
    \end{equation}
    for $x,\xi \in \R$ and $v, h \in H^1(\R)$.
  \end{enumerate}
\end{lem}
\begin{proof}
  For~\eqref{trafoPropC}, we estimate 		
  \begin{align*}
    \norm{\theta_x u - \theta_\xi v}{H^1(\R)} &\le  \norm{\theta_x (u - \theta_{\xi-x} v)}{H^1(\R)} = \norm{u - \theta_{\xi-x} v}{H^1(\R)} \le \\
    & \le \norm{u-v}{H^1(\R)} + \norm{v - \theta_{\xi -x} v}{H^1(\R)},
  \end{align*}
  where by the strong continuity of $\theta_x$ on $H^1(\R)$ the latter term vanishes as $|\xi - x| \to 0$.
  To show~\eqref{trafoPropC1}, we first verify that \eqref{eq:frechet} gives the Fr\'echet derivate of $F$, by estimating
  \begin{align*}
    R^2 &:=  \norm{(F(v+h,x+\xi) - F(v,x) - D_{(v,x)}F(\xi,h)}{L^2(\R)}^2 =  \\
    &= \norm{\theta_{x + \xi}(v + h) - \theta_x v - \xi \theta_x v' - \theta_x h}{L^2(\R)}^2 \le \\ 
    &\le 2  \norm{\theta_{\xi}v - v - \xi v'}{L^2(\R)}^2 + 2  \norm{\theta_{\xi}h - h}{L^2(\R)}^2 = \\
    &= 2 \int_{-\infty}^\infty \left(v(y + \xi) - v(y) - \xi v'(y)\right)^2 \d y + 2 \int_{-\infty}^\infty \left(h(y + \xi) - h(y)\right)^2 \d y.
  \end{align*}
  Applying first the fundamental theorem of calculus and in the second step Jensen's inequality and Fubini's theorem we continue with 
  \begin{align*}
    R^2 &\le 2 \xi^2 \left\{\int_{-\infty}^\infty \left( \int_0^1 \left( v'(y + z \xi) - v'(y)\right) \d z \right)^2 \d y + \int_{-\infty}^\infty \left( \int_0^1 h'(y + z \xi) \d z \right)^2 \d y \right\} \le \\
    & \le 2 \xi^2 \left\{ \int_0^1 \norm{ \theta_{z\xi} v' - v'}{L^2(\R)}^2 \d z + \norm{h'}{L^2(\R)}^2 \right\} \le \\
    &\le o \left(|\xi|^2\right) + 2 \xi^2 \norm{h}{H^1(\R)}^2, 
  \end{align*}
  showing Fr\'echet differentiability of $F$. To show continuous differentiability we estimate
  \begin{align*}
    \norm{D_{(v,x)}F(\xi,h) - D_{w,y}F(\xi,h)}{L^2(\R)} &= \norm{\xi \theta_x v' + \theta_x h - \xi \theta_y w'  - \theta_y h}{L^2(\R)} \le \\
    &\le |\xi| \norm{\theta_{x-y} v' - w'}{L^2(\R)} + \norm{\theta_{x-y}h - h}{L^2(\R)} \le \\
    &\le |\xi| \norm{\theta_{x-y} v' - w'}{L^2(\R)} + |x-y| \norm{h'}{L^2(\R)} \le \\
&\le \left(|\xi| + \norm{h}{H^1(\R)}\right) \max\left(\norm{\theta_{x-y}v' - w'}{L^2(\R)},|x-y|\right).
\end{align*}
Writing $\norm{.}{op}$ for the operator norm from $H^1(\R)\oplus \R$ to $L^2(\R)$, this shows that 
\[\norm{D_{(v,x)}F - D_{(w,y)}F}{op} \le \max\left(\norm{\theta_{x-y}v' - w'}{L^2(\R)},|x-y|\right).\]
The right hand side goes to zero as $ \norm{v-w}{H^1(\R)} +|x-y|\to 0$ by Lemma~\ref{lem:uniform_estimate}, and hence $D_{(v,x)}F$ depends continuously on $(v,x)$.
\end{proof}

\begin{thm}[Stochastic Chain Rule]\label{thm:chain_rule}
Let $(v,x)$ be given by \eqref{eq:SDEv} and set $u_t = F(v_t,x_t) = \theta_{x_t} v_t$. Then $u$ satisfies
\begin{equation}\label{eq:chain_rule}
      u_t = u_0 + \int_0^t \left(\left(\theta_{x_s}\mu_s\right) + \dot x_s (\theta_{x_s} v'_s)  \right) \d s + \int_0^t \left(\theta_{x_s}\sigma_s\right) \d W_s,
\end{equation}
on $\llbrak 0,\tau\llbrak$, where the first integral is an $L^2$-Bochner integral and the second one an $\HS$-stochastic integral.
\end{thm}
\begin{proof} Without loss of generality assume $\tau = \infty$ almost surely. Else, take an announcing sequence of stopping times $(\tau_N)$ for $\tau$. Then, multiply $\dot x$, $\mu_t$ and $\sigma_t$ with the indicator function $\1_{\llbrak 0,\tau_N\rrbrak}$and replace resp. $x_*$ and $v$ by $x_{.\wedge \tau_N}$ and $v_{.\wedge \tau_N}$. In this proof we use the shorthand notation $x_{s,t} = x_t - x_s$, $v_{s,t} = v_t - v_s$ etc.{}
For $\cP$ an arbitrary partition of $[0,T]$ we decompose
\begin{align*}
F(v_T,x_T) - F(v_0,x_0) &= \sum_{[s,t] \in \cP} F(v_t,x_t) - F(v_s, x_s) = \\
&=  \sum_{[s,t] \in \cP} D_{(x_s,v_s)} F(v_{s,t}, x_{s,t}) + R(v_s,x_s,v_{s,t},x_{s,t}).
\end{align*}
where $DF$ is the Fr\'echet derivative of $F$ from \eqref{eq:frechet} and $R$ is a remainder term.
Using equation \eqref{eq:SDEv} and the explicit forms of $F$ and $DF$ we rewrite
\begin{align*}
D_{(v_s,x_s)} F(v_{s,t}, x_{s,t}) &=  \theta_{x_s} v_{s,t}+x_{s,t} \theta_{x_s} v'_s = \\
&= \int_s^t \left(\theta_{x_s} \mu_r \right) \d r + \int_s^t \dot x_r \left(\theta_{x_s} v'_s\right)  \d r  + \int_s^t \left(\theta_{x_s} \sigma_r\right) \d W_r,\\
R(v_s,x_s,v_{s,t},x_{s,t}) &= \theta_{x_t} v_t - \theta_{x_s} v_t - x_{s,t} \left(\theta_{x_s} v'_s \right).
\end{align*}
The formula~\eqref{eq:chain_rule} follows, if we can show each of the following (a.s.) convergence statements
\begin{align}
\label{eq:shift_convergence} \sum_{[s,t] \in \cP_n} \int_s^t \dot x_r \left(\theta_{x_s} v'_s\right)  \d r &\quad \to \quad \int_0^T \dot x_r \left(\theta_{x_r} v'_r\right) \d r\\
\label{eq:mu_convergence} \sum_{[s,t] \in \cP_n} \int_s^t \left(\theta_{x_s} \mu_r \right)  \d r &\quad \to \quad \int_0^T \left(\theta_{x_r} \mu_r \right) \d r\\
\label{eq:remainder_convergence} \sum_{[s,t] \in \cP_n} \left(\theta_{x_t} v_t - \theta_{x_s} v_s - x_{s,t} \left(\theta_{x_s} v'_s \right)\right) & \quad \to \quad 0\\
\intertext{and}
\label{eq:sigma_convergence} \sum_{[s,t] \in \cP_n} \int_s^t \left(\theta_{x_s} \sigma_r \right)  \d W_r &\quad \to \quad \int_0^T \left(\theta_{x_r} \sigma_r \right) \d W_r
\end{align}
for some sequence of partitions $(\cP_n)_{n \in \N}$ as $n \to \infty$.\\
For claim \eqref{eq:shift_convergence} chose $\epsilon' > 0$ and set $\epsilon = \epsilon' / V_x[0,T]$ where $V_x[0,T]$ is the total variation of the process $x$ over the interval $[0,T]$. By Lemma~\ref{lem:uniform_estimate} we can find $\delta > 0$ such that
\[\norm{\theta_h v'_s - v'_r}{L^2(\R)} \le \epsilon\]
for all $|h| \le \delta, |s - r| \le \delta$. By continuity of $x$ we can choose the mesh of $\cP_n$ fine enough such that $|x_{r,s}| \le \delta$ for $|s - r| \le |\cP_n|$ and hence
\begin{multline*}
\norm{\sum_{[s,t] \in \cP_n} \int_s^t \dot x_r \left(\theta_{x_s} v'_s - \theta_{x_r} v'_r \right)  \d r}{L^2(\R)} \le \\
\le \sum_{[s,t] \in \cP_n} \int_s^t |\dot x_r| \norm{\theta_{x_{r,s}} v'_s - v'_r}{L^2(\R)}\d r \le V_x[0,T] \epsilon = \epsilon'
\end{multline*}
for $|\cP_n|$ small enough. As $\epsilon' > 0$ was arbitrary this shows \eqref{eq:shift_convergence} for any sequence of partitions with $|\cP_n| \to 0$. For \eqref{eq:mu_convergence} we estimate
\begin{equation*}
\norm{\sum_{[s,t] \in \cP_n} \int_s^t \left(\theta_{x_s} \mu_r - \theta_{x_r} \mu_r \right)  \d r}{L^2(\R)} \le \sum_{[s,t] \in \cP_n} \int_s^t \norm{\theta_{x_s - x_r} \mu_r - \mu_r}{L^2(\R)}\d r.
\end{equation*}
The integrand converges to $0$ for $\abs{\cP_n}\searrow 0$, due to strong continuity of $\theta$ and continuity of $x$. By dominated convergence, this carries over to the whole integral. 

For \eqref{eq:remainder_convergence} estimate 
\begin{align*}
R &:= \norm{\sum_{[s,t] \in \cP_n} \left(\theta_{x_t} v_t - \theta_{x_s} v_t - x_{s,t} \left(\theta_{x_s} v'_s \right) \right)}{L^2(\R)} \le \\
&\le \sum_{[s,t] \in \cP_n} \norm{\theta_{x_{s,t}} v_t - v_t - x_{s,t} v'_s}{L^2(\R)} = \\
&= \sum_{[s,t] \in \cP_n} \left(\int_{-\infty}^\infty \left(v_t(y + x_{s,t}) - v_t(y) - x_{s,t} v'_s(y) \right)^2 \d y\right)^{1/2}.
\end{align*}
 Applying first the fundamental theorem of calculus and in the second step Jensen's inequality and Fubini's theorem we continue with 
\begin{align*}
R &= \sum_{[s,t] \in \cP_n} |x_{s,t}| \left(\int_{-\infty}^\infty \int_0^1 \left(v'_t(y + z x_{s,t}) - v'_s(y) \right)^2 \d y \d z \right)^{1/2} \le \\
&\le \sum_{[s,t] \in \cP_n} |x_{s,t}| \left(\int_0^1 \norm{\theta_{zx_{s,t}} v_t' - v_s'}{L^2(\R)}^2 \d z \right)^{1/2}.
\end{align*}
Again $\sum_{[s,t] \in \cP_n} |x_{s,t}|$ can be bounded by the total variation of $x$, while Lemma~\ref{lem:uniform_estimate} shows that the $L^2$-norm vanishes uniformly as $|\cP_n| \to 0$. Finally, to obtain the convergence of the stochastic integrals in \eqref{eq:sigma_convergence}, we define the $\HS(U,L^2(\R))$-valued functions
\[ \Phi_{\cP_n}(r) = \sum_{[s,t]\in \cP_n} (\theta_{x_s}\sigma_r) \1_{(s,t]}(r),\qquad\Phi(r) = \theta_{x_r} \sigma_r\]
and note that \eqref{eq:sigma_convergence} is equivalent to $\lim_{n\to \infty} \int_0^T \Phi_{\cP_n}(r)\d W_r = \int_0^T \Phi(r)\d W_r$ along the sequence of partitions $\cP_n$.
Denoting by $\norm{.}{HS}$ the Hilbert-Schmidt-norm on $\HS(U,L^2(\R))$ we claim that 
\begin{equation}
\int_0^T \norm{\Phi_{\cP_n}(r) - \Phi(r)}{HS}^2 \d r \to 0,\qquad \text{as $|\cP_n| \to 0$.}
\end{equation}
Indeed, by rewriting
\begin{equation*}
\int_0^T \norm{\Phi_{\cP_n}(r) - \Phi(r)}{HS}^2 \d r = \sum_{[s,t] \in \cP_n} \int_s^t \norm{\theta_{x_s - x_r} \sigma_r - \sigma_r}{HS}^2 \d r
\end{equation*}
we may proceed as in the case of \eqref{eq:mu_convergence}, since $\theta$ is strongly continuous on $\HS(U;L^2(\R))$, see Remark~\ref{rmk:shiftHS}, to conclude that the Hilbert-Schmidt-norm on the right hand side vanishes uniformly as $|\cP_n| \to 0$. Choosing $\delta, \epsilon > 0$, \cite[Prop. 4.31]{dPZinf} says that 
  \begin{multline*}
    \PP\left[\norm{\int_0^t\Phi_{\cP_n}(r)-\Phi(r)\d W_r}{L^2(\R)} > \delta \right] \\
    \leq \epsilon + \PP\left[\int_0^t \norm{\Phi_{\cP_n}(r) -
        \Phi(r)}{HS}^2 \d r > \frac{\epsilon}{\delta^2} \right] \phantom{blabla}
  \end{multline*}
  and we see that as $|\cP_n| \to 0$ the right hand side can be made arbitrarily small. Hence, $\Phi_{\cP_n}(r)\d W_r \to \int_0^T \Phi(r)\d W_r$ in probability, as $|\cP_n| \to 0$. In particular any sequence of partitions with mesh tending to zero contains a subsequence $(\cP_{n_k})_{k \in \N}$ such that the convergence of stochastic integrals takes place almost surely, completing the proof of \eqref{eq:sigma_convergence}.
\end{proof}

\subsection{Completing the proof of the main result}
Making use of the notation defined in section~\ref{sec:sfbp} and~\ref{sec:proofs} we complete the proof of Theorem~\ref{thm:sfbp} by combining the relevant results.
\begin{proof}[Proof of Theorem~\ref{thm:sfbp}]

Given initial data $x_0\in \R$, $v_0\in \Gamma(x_0)$ we set
  \[X_0 := (v_0(x_0 + (.))\vert_{\R_{\geq 0}} , v_0(x_0 - (.))\vert_{\R_{\geq 0}}, x_0) \in \DA.\]
  By application of the Lemmas in section~\ref{sec:proofs} to Theorem~\ref{thm:EUCloc} and Corollary~\ref{thm:eus} there exists a unique maximal strong solution of~\eqref{eq:SEEq} $X$ on $\L^2$ with initial value $X_0$. Note that $X$ is $\H^2$ continuous and denote its $\H^2$-explosion time by $\tau$. 
Setting 
\[u_1(t) := X_1(t), \quad u_2(t) := X_2(t), \quad x_*(t) := X_3(t)\]
we obtain a solution of the fixed boundary problem \eqref{eq:fbpt} on $\llbrak 0,\tau\llbrak$. Next we paste together $u_1$ and $u_2$ by setting
\[u_t(x) = u_1(t,x) \1_{\R_+}(x) + u_2(t,-x) \1_{\R_-}(x)\]
and note that due to the Dirichlet boundary condition at $0$ it holds that $u_t \in H^1(\R)$. Thus we may apply the stochastic chain rule of 
Theorem~\ref{thm:chain_rule} to $v_t := F(u_t, -x_*(t)) = \theta_{-x_*(t)} u_t$ to obtain a local solution $v_t$ of the stochastic free boundary problem \eqref{eq:fbp}.

To show uniqueness we assume that there exists another local solution $(\check v,\check x_*)$ of \eqref{eq:fbp} on a stochastic interval $\llbrak 0,\varsigma \llbrak$, which, by definition is a.\,s. $H^1(\R)$-continuous. Applying the stochastic chain rule of 
Theorem~\ref{thm:chain_rule} to $\check u_t = F(\check v_t, x_*(t)) = \theta_{\check x_*(t)} \check v_t$ we obtain a local solution $\check u_t$ of the stochastic fixed boundary problem \eqref{eq:fbp}. Reversing the procedure from above $\check u_t$ can be rewritten as a solution $\check X(t)$ of the abstract stochastic evolution equation \eqref{eq:SEEq}. The parts of Theorem~\ref{thm:EUCloc} and Corollary~\ref{thm:eus} on uniqueness and maximality imply $\varsigma \leq \tau$ and that $\check X(t)$ is equal to the solution $X(t)$ constructed previously. We conclude that also $(\check v,\check x_*) = (v,x_*)$ and the proof is complete.
%
\end{proof}

%% file: nemytskii.tex
\section{The Nemytskii operator on Sobolev spaces}\label{A:nem}
In this section we prove some regularity results on the Nemytskii operator $N$,
\[ N(u)(x) = \mu(x,u(x)),\]
on the Sobolev spaces $H^k(\R_+)$. Here, $\mu :\R_+ \times \R^d \rightarrow \R$ and $x\in \R_+$. Note that these results are well-known, even for more general spaces, in the case of bounded domains, see e.g. \cite{valentBook}, \cite{appell}. However, in the case of unbounded domains several additional conditions on $\mu$ are necessary to make them work.
First, we state a result which guarantees, that under certain assumptions on $\mu$, $N$ maps $H^k$ into $H^k$. For a proof we refer to \cite[Theorem 1]{valentPaper}, of which it is a special case.
\begin{lem}\label{lem:mult}
  For each integer $k\geq 1$ the space $H^k(\R_+)$ is a Banach algebra. In particular, there exists a constant $c$ such that for all $u$, $v\in H^k(\R_+)$ it holds that $uv\in H^k(\R_+)$ and
  \[ \norm{uv}{H^k} \leq c \norm{u}{H^k}\norm{v}{H^k}.\]
\end{lem}
Next, we adapt \cite[Theorem 2]{valentPaper} to our setting. For notational reasons we also introduce the Nemytskii operators
\[ N_x(u)(x) := \left(\ddx \mu\right)(x,u(x)),\quad N_{y_j}(u)(x):=\left(\frac{\partial}{\partial y_j} \mu\right)(x,u(x)),\;j=1,...,d,\]
for $\quad u\in H^k(\R_+; \R^d),\;x\in \R_+$.
In order for $N$ to map $H^k$ into $H^k$ again, we need certain growths restrictions, which is not the case on bounded domains. 
\begin{ass}\label{a:mugrowths} Assume $\mu\in C^m(\R_{\geq 0}\times \R^d, \R)$ and
  \begin{enumerate}[label=(\alph*)]
  \item For each integer $l$, $0\leq l\leq m$ there exists an $ a_l\in L^2(\R_+)$ and some $ b_l:\R^d\rightarrow \R_+$ locally bounded, such that 
    \[\abs{D^{(l,0,...,0)} \mu(x,y)} \leq  a_l(x) +  b_l(y)\abs{y},\quad\forall\,x\in \R_+,\,y\in \R^d\] 
  \item For each multiindex $\alpha$ with $\alpha_1 \neq \abs{\alpha} \leq m$, the functions $\sup_{x\in \R_+} \abs{D^\alpha \mu (x,.)}$ are locally bounded.
  \end{enumerate}
\end{ass}
\begin{ass}\label{a:mulip}
  Assume that $\mu\in C^{m}(\R_{\geq 0} \times \R^d, \R)$ and $D^{\alpha}\mu(x,.)$ is locally Lipschitz for all multi-indices $\alpha$, $\abs{\alpha}\leq m$ with Lipschitz constants uniform in $x\in \R_{\geq 0}$, i.\,e. we assume that for all $r\geq 0$ there exists $L_r\geq 0$ such that   \begin{equation*}
    \abs{D^{\alpha}\mu(x,y) - D^{\alpha}\mu(x,z)} \leq L_r \abs{y-z}.
  \end{equation*}
  holds for all $y$, $z\in \R^d$ with $\abs{y}$, $\abs{z}\leq r$ and $\alpha$, $\abs{\alpha}\leq m$.
\end{ass}
\begin{rmk}
  If $\mu$ satisfies Assumption~\ref{a:mugrowths} for some integer $m\geq 1$, then $\mu$ satisfies Assumption~\ref{a:mulip} for $m-1$.
\end{rmk}
\begin{rmk}
  Recall the Sobolev embeddings
  \[H^{m+1}(\R_+) \hookrightarrow BUC^m(\R_+),\]
  where $BUC^m(\R_+)$ denotes the Banach space of functions with bounded and uniformly continuous derivatives up to order $m$. As usual $BUC^m(\R_+)$ is equipped with the $C^m$-norm. In the following, we will work with the $BUC^m$ representative of the elements in $H^{m+1}$ without further comment.
\end{rmk}
\begin{thm}\label{thm:Ncont}
  If Assumption~\ref{a:mugrowths} holds for some integer $m\geq 1$, then the operator $N$ is continuous from $(H^m(\R_+))^d$ into $H^m(\R_+)$.
\end{thm}
\begin{proof}
  We adapt the proof of \cite[Theorem 2]{valentPaper} for the domain $\R_+$ and the spaces $H^k$ by incorporating the additional growths assumptions. We proceed by induction and consider $m=1$, first. Since $\mu\in C^1(\R_{\geq 0}\times \R^d)$ we get immediately that $N(u)$, $N_x(u)$ and $N_{y_j}(u)$, $j=1,...,d$ are bounded and continuous functions for $u\in H^1(\R_+; \R^d)$ fixed. Let now $(u^n) \subset H^1(\R_+; \R^d) \cap C^\infty(\R_{\geq 0};\R^d)$ such that $u^n \longrightarrow u$ in $H^1$. Then the convergence also takes place in $\norm{.}{\infty}$ and by the chain rule we can write
  \begin{equation}
    \tfrac{\d}{\d x} N(u^n) = N_x(u^n) + \sum_{j=1}^n N_{y_j}(u^n) \nabla u^n_j. \label{eq:DN}
  \end{equation}
  By assumption,
  \[\abs{N(u^n)} \leq a_0 + (b_0 (u^n)) \abs{u^n},\text{ and }\abs{N_x(u^n)} \leq a_1 + (b_1(u^n)) \abs{u^n}.\]
  Since $u^n\rightarrow u$ uniformly and in $L^2$ we get for each estimate that $\norm{b_i(u^n)}{\infty}$ for $i \in \{0,1\}$ is bounded in $n\in \N$ such that both, $N(u^n)$ and $N_x(u^n)$ are bounded by $L^2$-converging sequences. Hence, we can apply a version of Lebesgue's dominated convergence~\cite[Theorem 1.21]{kallenberg} and obtain the $L^2$ convergence of
  \[ N(u^n) \longrightarrow N(u),\quad\text{and}\quad N_x(u^n) \longrightarrow N_x(u),\;\text{ as }n\rightarrow \infty.\]
  For the remaining summands we get
  \begin{multline*}
    \norm{N_{y_j} (u^n) \nabla u^n_j - N_{y_j} (u) \nabla u}{L^2} \leq \sup_{n\in\N}\norm{N_{y_j}(u^n)}{\infty} \norm{\nabla u^n_j - \nabla u_j}{L^2} \\ + \norm{\left( N_{y_j}(u^n) - N_{y_j}(u)\right) \nabla u_j}{L^2}
  \end{multline*}
  which goes to $0$ as $n\to \infty$. Indeed, uniform convergence of $(u^n)$ and dominated convergence yield $L^2$-convergence of $(N_{y_j}(u^n)\nabla u)$. Hence,
  \[ N(u^n) \xrightarrow[]{H^1}N(u),\quad\text{ as }n\to \infty.\]
  By completeness of $H^1$ this implies $N(u)\in H^1(\R_+)$ and also shows the continuity of $N$ for the case $m=1$.

For the induction step from $m$ to $m+1$ we may assume that the claim holds true for $m \geq 1$ and that Assumption~\ref{a:mugrowths} holds for $(m+1)$. Clearly, the assumption also holds for $m$ and so $N$ maps $H^{m+1}$ continuously into $H^m$ by induction hypothesis. It thus remains to show that also $\tfrac{\d}{\d x}\circ N$ maps $H^{m+1}$ into $H^m$. We decompose $\tfrac{\d}{\d x}\circ N$ as in \eqref{eq:DN} and note that Assumption~\ref{a:mugrowths} for $m$ is also satisfied by $\ddx \mu$ and by
  \begin{equation}
    \label{eq:mutransf}
(x,y,z) \mapsto \tilde\mu_j(x,y,z) := \tfrac{\partial}{\partial y_j} \mu(x,y)z.    
  \end{equation}
  Hence, by induction hypothesis the operators $N_x$ and $\tilde N_j$ are continuous from $H^{m}(\R_+)^d$, resp. $H^m (\R_+)^{d+1}$, into $H^m (\R_+)$, where $\tilde N_j$ is the Nemytskii operator defined by $\tilde \mu_j$, $j=1,...,d$. Since also $u\mapsto \nabla u$ is continuous from $H^{m+1}$ into $H^m$ and Lemma~\ref{lem:mult} shows continuity of multiplication, \eqref{eq:DN} yields continuity of $\tfrac{\d}{\d x}\circ N$ from $H^{m+1}$ into $H^m$, as claimed.
\end{proof}
\begin{thm}\label{thm:Nlip}
  Let $\mu$ satisfy Assumptions~\ref{a:mugrowths} and~\ref{a:mulip} for some positive integer $m$. Then, $N$ is Lipschitz continuous from bounded subsets of $(H^m(\R_+))^d$ into $H^m(\R_+)$.
\end{thm}
\begin{proof}
  We proceed as above, by induction on $m\in \N$. First, let $m=1$ and $u$, $v\in H^1(\R_+;\R^d)$ with $\norm{u}{H^1}$, $\norm{v}{H^1} \leq r$. By continuity of $N$ we can assume, w.\,l.\,o.\,g. $u$, $v\in H^1(\R_+;\R^d)\cap C^\infty(\R_{\geq 0}; \R^d)$. By Sobolev embeddings there exists a constant $c$ s.\,t. $\norm{u}{\infty}$,$\norm{v}{\infty} \leq cr$. By Assumption~\ref{a:mulip},
  \begin{align*}
    \norm{N(u) - N(v)}{L^2} &\leq L_{cr} \norm{u-v}{L^2(\R_+;\R^d)},\\
    \norm{N_x(u) - N_x(v)}{L^2} &\leq L_{cr} \norm{u-v}{L^2(\R_+;\R^d)},
  \end{align*}
  and for $j=1,...,d$,
  \begin{multline}
    \label{eq:27}
    \norm{N_{y_j}(u) \nabla u_j - N_{y_j}(v) \nabla  v_j}{L^2}\\ \leq \norm{N_{y_j}(u)}{\infty} \norm{\nabla u_j - \nabla  v_j}{L^2} + \norm{N_{y_j}(u) - N_{y_j}(v)}{\infty} \norm{\nabla v_j}{L^2}\\
    \leq K_{j,cr} \norm{\nabla u_j - \nabla v_j}{L^2} + L_{cr}\norm{u_j - v_j}{L^2},
  \end{multline}
  for
  \(K_{j,cr} := \sup_{\abs{y}\leq cr} \sup_{x\in \R}\abs{\ddy_j \mu(x,y)}<\infty. \) Chain rule~\eqref{eq:DN} then yields the assertion for $m=1$. 

  For the induction step we may assume that the theorem holds for fixed $m$ and that Assumptions~\ref{a:mugrowths} and~\ref{a:mulip} are satisfied for $m+1$. By induction hypothesis, $N$ is Lipschitz on bounded sets from $H^m(\R_+)^d$ into $H^m(\R_+)$ and thus, also from $H^{m+1}(\R_+)^d$ into $H^m(\R_+)$. Hence, it suffices to show that $\tfrac{\d}{\d x}\circ N$ is Lipschitz on bounded sets from $H^m(\R_+)^d$ into $H^m(\R_+)$. To this end note that $\ddx \mu$ satisfies Assumptions~\ref{a:mugrowths} and~\ref{a:mulip} as well as $\tilde \mu_j$, $j=1,...,d$, defined in~\eqref{eq:mutransf}. By induction hypothesis, the operators $N_x$ and $\tilde N_j$, $j=1,...,d$, defined in the proof of Theorem~\ref{thm:Ncont}, are Lipschitz on bounded sets. Again, approximation by elements in $H^m\cap C^m$ and~\eqref{eq:DN} then show that the same holds true for $\tfrac{\d}{\d x}\circ N$.
\end{proof}

%% file: noisered.tex
\section{The noise operator}\label{A:noise}
In this section we will study the operator-valued map $\cC$, defined in~\eqref{eq:coefC} by
 \[ (\cC(u)w)(x) =
 \begin{pmatrix}
   \sigma_+(x,u_1(x)) (T_\zeta w)(u_3+x)\\ \sigma_-(-x,u_2(x)) (T_\zeta w)(u_3-x) \\ 0
 \end{pmatrix}
\]
 for $u\in \DA$, $w\in U$ and $x\in \R$. We can reduce the problem to the operator
 \begin{equation}
   \label{eq:Ccomponent}
   (u,x_*)\mapsto \sigma(.,u(.)) T_\zeta (.+x_*)
 \end{equation}
 for $\sigma$ satisfying Assumption~\ref{a:sigma} and $\zeta$ as in Assumption~\ref{a:zeta}. Define the Nemytskii operator 
 \[ N_\sigma: H^2(\R_+) \to H^2(\R_+),\; u\mapsto \sigma(.,u(.)),\]
 which is Lipschitz on bounded sets by Theorem~\ref{thm:Nlip}.
 \begin{lem}\label{lem:HtimesC}
   Multiplication is bilinear continuous from $H^2(\R_+) \times BUC^2(\R_{\geq 0})$ into $H^2(\R_+)$. 
 \end{lem}
 \begin{proof}
   By density of $C^n \cap H^n$ in $H^n$, $n\in\N$, one can check that Leibniz formula holds for multiplication on $H^n\times BUC^n$, so that
  \[ D^k(uf) =  \sum_{j=0}^k \binom{k}{j} D^{j}u f^{(k-j)} \]
  which is clearly square integrable for $u\in H^n(\R_+)$, $f\in BUC^n(\R_{\geq 0})$ and $k\leq n$. In particular, for $n=2$,
  \[ \abs{D^k(uf)(x)} \leq K \norm{f}{C^2} \sum_{j=0}^k D^{j}u(x),\quad k\leq 2\]
  for some constant $K$, so that for some $\tilde K$,
  \begin{equation} \norm{uf}{H^2(\R_+)} \leq \tilde K \norm{f}{C^2}\norm{u}{H^2}.  \qedhere\end{equation}
\end{proof}
\begin{lem}\label{lem:Tc}
  $T_\zeta$ maps $U$ into $BUC^2(\R)$. Moreover, $T_\zeta w$ and its first two derivatives are Lipschitz continuous for all $w\in U$.
\end{lem}
\begin{proof}
  First note that for all $x,y\in \R$ it holds that 
  \[ \sup_{-1<\epsilon <1}\abs{ \ddx \zeta^{(i)}(x+\epsilon,y)} \leq \int_{-1}^1 \abs{\zeta^{(i+1)}(x+\epsilon',y)} \d \epsilon' + \abs{\zeta^{(i)}(x,y)},\quad i=0,...,2\]
  so that Lebesgue's differentiation lemma gives 
  \begin{equation}
    \tfrac{\d}{\d x} T_{\zeta^{(i)}}w(x) = T_{\zeta^{(i+1)}}w(x),\quad x\in \R,\;w\in U.\label{eq:IntDiff}
  \end{equation}
  Hence, it suffices to show that $T_\zeta w\in BUC(\R)$ provided that~\eqref{eq:Azeta} holds for $i=0$ and $i=1$. For fixed $w\in U$ and $x_1$, $x_2\in\R$ we directly get
  \begin{equation}
    \label{eq:TLip}
    \begin{split}
      \abs{T_\zeta w(x_1) - T_\zeta w(x_2)} &\leq \int_\R \int_0^1 \abs{ \zeta'(x_2 + \epsilon (x_1-x_2), y)} \abs{w(y)} \d \epsilon \d y \abs{x_1-x_2} \\
      &\leq \sup_{x\in \R} \norm{\zeta'(x,.)}{L^2(\R)} \norm{w}{L^2} \abs{x_1-x_2}.
    \end{split}
  \end{equation}
  Here, we used fundamental theorem of calculus, Tonelli's theorem and the Cauchy-Schwartz inequality. Hence, $T_\zeta w$ is globally Lipschitz and particularly uniformly continuous. Analogously, Cauchy-Schwartz inequality yields
  \begin{equation}
    \sup_{x\in \R}\abs{T_\zeta w(x)} \leq \sup_{x\in \R} \norm{\zeta(x,.)}{L^2(\R)} \norm{w}{U} <\infty.\label{eq:Tcnorm}  \qedhere
  \end{equation}
\end{proof}
\begin{rmk}\label{rmk:TzL}
  From~\eqref{eq:Tcnorm} we immediatly get $T_\zeta \in L(U, BUC^2(\R))$. However, in general $T_\zeta$ itself is not Hilbert-Schmidt. To get the Hilbert-Schmidt property we need the multiplication with $N_\sigma$ as we will show in the next lemma.
\end{rmk}
\begin{lem}\label{lem:HSest}
  Let $\Delta$ be the Dirichlet Laplacian on $L^2(\R_+)$. For $u\in \dom(\Delta)$ and $x_*\in \R$ it holds that
  \[\norm{u \* (\theta_{x_*}\circ T_\zeta(.)))}{\HS(U; \dom(\Delta))} \leq K \norm{u}{\Delta} \sup_{x\in \R}\sum_{i=0}^2\norm{\zeta^{(i)}(x,.)}{L^2}\]
\end{lem}
\begin{rmk}
  This result immediately extends to $\cC(u)$, because Assumption~\ref{a:sigma} and Lemma~\ref{lem:HtimesC} assure $N_\sigma(u)\in \dom(\Delta)$ for all $u\in \dom(\Delta)$. Moreover, note that
  \[ \theta_x\circ T_{\zeta} =   T_{\zeta_x},\]
  where $\zeta_x := \zeta(x+.,.)$ satisfies Assumption~\ref{a:zeta}, too.
\end{rmk}
\begin{proof}
  Linearity and continuity in $w$ follow directly from the construction and Remark~\ref{rmk:TzL} and we are now interested in the Hilbert Schmidt norm. Without loss of generality, we can choose $x_* = 0$. So denote by $(e_k)$ an arbitrary CONS of $U$, then
  \begin{equation}
    \norm{u\*(T_\zeta (.))}{\HS(U,\dom(\Delta))}^2 = \sum_{k=1}^\infty     \norm{u\* (T_\zeta e_k)}{L^2}^2 +     \norm{\Delta(u \* (T_\zeta e_k))}{L^2}^2,
  \end{equation}
  and the first sum equals
  \begin{multline}
    \sum_{k=1}^\infty \int_{\R_+} u(x)^2 \langle \zeta(x, .), e_k\rangle_{L^2(\R)}^2\\
    = \int_{\R_+} u(x)^2 \norm{\zeta(x,.)}{L^2(\R)}^2 
    \leq \norm{u}{L^2(\R_+)}^2 \sup_{x\in \R}\norm{\zeta(x,.)}{L^2(\R)}^2,
  \end{multline}
  where we used Tonelli's theorem and Parseval's identity for the first equality. To bound the second sum we proceed on exactly the same way but first apply Leibnitz rule to get the second (weak) derivative
  \begin{multline}
    \sum_{k=1}^\infty \sum_{i=0}^2 \binom{2}{i} \int_{\R_+}\abs{\tfrac{\partial^{i}}{\partial x^i}u(x)}^2\ \langle \zeta^{(2-i)}(x,.),e_k\rangle_{L^2(\R)}^2 \\
    \leq \sum_{i=0}^2  \binom{2}{i}\norm{\tfrac{\partial^{i}}{\partial x^i} u}{L^2(\R_+)}^2 \sup_{x\in \R}\norm{\zeta^{(2-i)}(x,.)}{L^2(\R)}^2\\
    \leq 4 \norm{u}{H^2(\R_+)}^2 \sup_{x\in\R}\sum_{i=0}^2\norm{\zeta^{(i)}(x,.)}{L^2(\R)}^2.
  \end{multline}
  By equivalence of $\norm{.}{H^2}$ and $\norm{.}{\Delta}$ the result follows.
\end{proof}

To show the main result of this appendix, we just need to combine the previous lemmas.
\begin{thm}\label{thm:CLip}
  The map $\cC:\DA \rightarrow \HS(U,\DA)$ is Lipschitz continuous on bounded sets.
\end{thm}
\begin{proof}
  By the structure of $\cA$ and $\L^2$, it suffices to show the property for the operator defined in~\eqref{eq:Ccomponent}. 
  Assumption~\ref{a:sigma} yields $N_\sigma(\dom(\Delta))\subset \dom(\Delta) $ and we can apply Lemma~\ref{lem:HSest}. For $u$, $\tilde u\in \dom(\Delta)$, $x_*$, $y_*\in \R$ and writing $ \zeta_{z,\tilde z}(x,y) := \zeta(z+x,y) - \zeta(\tilde z + x)$, it holds that
  \begin{gather*}
    \begin{split}
      &\norm{N_\sigma(u)(\theta_{x}T_\zeta) - N_\sigma(\tilde u) (\theta_{y}T_\zeta) }{\HS} \\
      &\qquad \leq \norm{(N_\sigma(u) - N_\sigma(\tilde u))\* (\theta_{x} T_\zeta .)}{\HS} + \norm{N_\sigma(\tilde u) \* (\theta_0 T_{ \zeta_{x,y}} .)}{\HS}\\
      &\qquad \leq K_\zeta \norm{N_\sigma(u) - N_\sigma (\tilde u)}{\Delta} + K\norm{N_\sigma(\tilde u)}{\Delta} \sup_{z\in \R} \sum_{i=0}^2 \norm{ \zeta_{x,y}^{(i)}(z,.)}{L^2}.
    \end{split}
  \end{gather*}
  A computation similar to~\eqref{eq:TLip} shows
  \begin{equation*}
    \sup_{z\in \R}\norm{\zeta_{x,y}^{(i)}(z,.)}{L^2}^2     
    \leq \abs{x-y}^2 \sup_{z\in \R}\norm{\zeta^{(i+1)}(z,.)}{L^2}^2.
  \end{equation*} 
  Finally, we put everything together and use that on bounded sets $N_\sigma$ is Lipschitz, and thus bounded, to get the assertion.
\end{proof}